\documentclass[10pt]{amsart}
\usepackage{amsmath,amssymb}
\calclayout
\usepackage{enumerate}
\usepackage{color}
\usepackage{hyperref}
\usepackage{natbib}

\usepackage{graphicx}
\usepackage{dsfont}

\newtheorem{theorem}{Theorem}[section]
\newtheorem{proposition}[theorem]{Proposition}
\newtheorem{lemma}[theorem]{Lemma}
\newtheorem{corollary}[theorem]{Corollary}
\newtheorem{definition}[theorem]{Definition}
\newtheorem{remark}[theorem]{Remark}
\newtheorem{example}[theorem]{Example}
\newtheorem{assumption}[theorem]{Assumption}

\newcommand{\R}{\mathbb{R}}
\newcommand{\Q}{\mathbb{Q}}

\newcommand{\N}{\mathbb{N}}
\newcommand{\E}{\mathbb{E}}

\renewcommand{\P}{\mathds{P}}

\newcommand{\Dc}{\mathcal{D}}
\renewcommand{\epsilon}{\varepsilon}

\newcommand{\x}{x_0}

\begin{document}

\title[]{Strategies with minimal norm are optimal for expected utility maximization under high model ambiguity}

\date{\today}

\author{Laurence Carassus}
\address{Laurence Carassus \newline L\'{e}onard de Vinci P\^ole Universitaire, Research Center, 92 916 Paris La D\'{e}fense, France \newline
and LMR, UMR 9008, Universit\'e Reims Champagne-Ardenne, France}
\email{laurence.carassus@devinci.fr}
\author{Johannes Wiesel}
\address{Johannes Wiesel\newline
Columbia University, Department of Statistics\newline
1255 Amsterdam Avenue\newline
New York, NY 10027, USA}
\email{johannes.wiesel@columbia.edu}

\thanks{JW acknowledges support by NSF Grant DMS-2205534.}

\begin{abstract}
We investigate  an expected utility maximization problem under model uncertainty in a one-period financial market. We capture model uncertainty by replacing the baseline model $\mathbb{P}$ with an adverse choice from a Wasserstein ball of radius $k$ around $\mathbb{P}$ in the space of probability measures and consider the corresponding Wasserstein distributionally robust optimization problem. We show that optimal solutions converge to a strategy with minimal norm when uncertainty is increasingly large, i.e. when the radius $k$ tends to infinity.
\end{abstract}

\maketitle

\vspace{2cm}

\section{Introduction}
In this article we consider a one period financial market with $d$ risky assets and a num\'{e}raire, all of them being traded.  We refer to Section \ref{sec:finance} for a more detailed description of the financial market under consideration.
We investigate the \emph{Wasserstein distributionally robust optimization problem}
\begin{align}
\label{eq:valeur}
u(\x):=\sup_{w \in D }\inf_{\Q\in B_k(\P)} \E_\Q[U(\x+\x\langle w, X\rangle)]
\end{align}
for a concave non-decreasing utility function $U:\R\to \R,$ a non-empty, closed constraints set $D\subseteq \R^d$ and an initial capital $\x\neq 0$. In the above, $\langle \cdot, \cdot \rangle$ denotes the scalar product on $\R^d$ and the vector $w\in \R^d$ denotes the portfolio weights in the $d$ risky assets. 
The probability measures $\P$ and $\Q$ describe the laws of the vector  of returns of the $d$ risky assets, and  $X:\R^{d}\to \R^{d}$ denotes the identity map.  
We assume that $\E_{\P}[|X|^p]<\infty$ for some $p\ge 1,$  where $|\cdot|$ is the Euclidean norm on $\R^d$, and denote
the $p$-Wasserstein ball of radius $k>0$ around $\P$ by $B_k(\P)$. 
Defining 
\begin{align}\label{eq:prob}
u(k,w):= \inf_{\Q\in B_k(\P)} \E_\Q[U(\x+\x\langle w, X\rangle)]
\end{align}
for all  $w\in \R^d$ and   $$w_k:=\mathrm{arg\ max}_{w \in  D} u(k,w),$$
the main result of this paper is the following:

\begin{theorem}\label{thm:main}
Let $U: \R \to \R$ be a non-decreasing, non-constant, differentiable and concave function and assume that
there exist  some constants $C_1,\underline{x}>0$
such that $U(-\underline{x})<0$ and for all $x\leq - \underline{x}$  
\begin{align}\label{mino}
U(x) & \ge -C_1(1+|x|^p).
\end{align}
Furthermore assume that one of the three conditions holds true:
\begin{enumerate}[(i)]
\item $U$ is bounded from above.
\item $\mbox{AE}_{-\infty}(U):=\liminf_{x \to -\infty} \frac{xU\rq{}(x)}{U(x)} >1$.
\item $\mbox{AE}_{+\infty}(U):=\limsup_{x \to +\infty} \frac{xU\rq{}(x)}{U(x)} <1.$
\end{enumerate}
Then $w_k$ is well-defined and $\lim_{k\to \infty} w_k$ exists and belongs to $\Dc$, where 
\begin{align}\label{eq:wu1}
\Dc=\{\overline{w} \in D\;:  \; |\overline{w}|= \inf_{w \in D}|w|\}.
\end{align}
\end{theorem}
\begin{remark}
\begin{enumerate}[(a)]
\item If  \eqref{mino} is not postulated, the problem is ill-posed (see Lemma \ref{lem:ill}).  The conditions (ii) and (iii) of Reasonable Asymptotic Elasticity in Theorem \ref{thm:main} are classical, see the discussion in Section \ref{sec:discussion}. One may instead assume the milder conditions \eqref{conc}, see 
Lemma \ref{lem:bc2} and Proposition \ref{Kbound} below. 
\item If $0\in  D$, then trivially $\Dc=\{0\}$. Assume now that $0\notin  D.$ As $D$ is closed, 
$\Dc\neq \emptyset$ and $\Dc$ is included in $\partial D= D \setminus \mathring{D}$, the boundary of $D$. If $D$ is furthermore assumed to be convex, then $\Dc$ reduces to a singleton.
\item In the statement of Theorem \ref{thm:main}, one can replace the Euclidean norm $|\cdot|$ by any other norm $\|\cdot\|$ on $\R^d$. In that case, the Wasserstein ball $B_k(\P)$  has to be defined wrt. the dual norm $\|\cdot\|_\ast := \sup_{\|w\|\le 1} \langle \cdot, w\rangle$ instead of $|\cdot|$; see \eqref{eq:wass} below.
\item We prove in Theorem \ref{thm:main} that $\lim_{k \to \infty}  \mathrm{arg\ max}_{w \in D} u(k,w)\in \Dc$. 
 One may wonder if $\mathrm{arg\ max}_{w \in D} \lim_{k \to \infty} u(k,w) \in \Dc$ too. We show in Lemma \ref{lem:-infty} that in fact $\lim_{k \to \infty} u(k,w)=-\infty,$ so that any $w$ is an optimizer. 
\end{enumerate}
\end{remark}

\begin{example}
We now provide examples of $D$ and $U$, for which the assumptions made in Theorem \ref{thm:main} are satisfied. We start with the constraints set $D$.
\begin{enumerate}[(a)]
\item Let $a>0$ and set
\begin{align*}
D=\{w\in \R^d: \langle w, \mathds{1}\rangle\ge a\} \mbox{ or } D=\{w\in \R_+^d: \langle w, \mathds{1}\rangle\ge a\},
\end{align*}
where $\mathds{1}=(1,\dots, 1)$   is the unit vector in $\R^d$. 
Then,  in both cases, $D$ is closed, convex and $\Dc=\{w^u\}$, where $w^u=\frac{a}{d} \mathds{1}$ is the rescaled unit vector satisfying $\langle w^u, \mathds{1}\rangle =a$. 
So, when we work under the constraint that the total exposure to the risky assets is positive (with or without no short-selling constraints)  we find that optimal solutions converge to the uniform diversification strategy when uncertainty is increasingly large. As $U$ is defined on $\R$, the investor may have negative wealth and  we may also allow negative total exposure to the risky assets, i.e. 
\begin{align*}
D=\{w\in \R^d: \langle w, \mathds{1}\rangle\ge a  \mbox{ or }\langle w,  \mathds{1}\rangle\le -a\}. 
\end{align*}
As $D$ is not convex we find that $\Dc$ is not a singleton any more and $\Dc=\{w^u, -w^u\}.$
\item We now give examples of $U$ satisfying our assumptions. Recall that $p \geq 1$ and define
\begin{align*}
U(x):=
(1-e^{-x})
\mathds{1}_{\{x>0\}} -\frac{1}{p}((1-x)^{p}-1)\mathds{1}_{\{x\leq 0\}}.
\end{align*}
 Then $U$ is concave, strictly increasing, continuously differentiable, bounded from above and
it satisfies  \eqref{mino}. Next, if
\begin{align*}
U(x):=\log(x)\mathds{1}_{\{x>1\}} +(x-1)\mathds{1}_{\{x\leq 1\}},
\end{align*}
then $U$ is concave, strictly increasing, continuously differentiable and
it satisfies \eqref{mino} and (iii) with $\mbox{AE}_{+\infty}(U)=0$. 
Lastly, (ii) can only occur if $p>1$, see Remark \ref{rem2}. So let us assume that $p>1$ and take $q \in (1,p]$ and define 
\begin{align*}
U(x):=x
\mathds{1}_{\{x>0\}} -\frac{1}{q}( (1-x)^{q}-1) \mathds{1}_{\{x\leq 0\}}.
\end{align*}
Then $U$ is concave, strictly increasing, continuously differentiable and
it satisfies  \eqref{mino} and (ii) with $\mbox{AE}_{-\infty}(U)=q$. 
\end{enumerate}
\end{example}

\section{Literature review and discussion of Theorem \ref{thm:main}}

The study of optimal investment problems under a fixed probabilistic model $\P$ goes back at least to \cite{merton1969lifetime}. More recently, it has been widely acknowledged that agents might face some degree of uncertainty about the choice of $\P$. This is often referred to as \emph{Knightian uncertainty} after \cite{knight2012risk} and has led to the max-min approach to optimal investment, as studied in detail in \cite{gilboa1989maxmin,maccheroni2006ambiguity,schied2007optimal,schied2009robust}. These works consider an agent who is averse to model uncertainty and thus solves the worst utility maximization problem 
$$ \max_{w} \min_{\Q\in \mathcal{Q}}\E_{\Q}[U(\x+\x\langle w, X\rangle)];$$ we refer to \cite{obloj2021distributionally} for a more complete discussion with historical references.\\

There are many ways to model Knightian uncertainty through the ambiguity set $\mathcal{Q}$ above: we refer to \cite{maenhout2004robust,calafiore2007ambiguous,kerkhof2010model,hansen2011robustness}, and the references therein. In this article we define $\mathcal{Q}$ to be a ball $B_k(\P)$ in the \textit{Wasserstein distance} $\mathcal{W}_p$ around $\P$. The Wasserstein distance has seen a recent spike of interest in the mathematical finance and statistics literature. One reason for this is the fact that Wasserstein balls can be considered non-parametric neighbourhoods of $\P$; indeed $B_k(\P)$ contains infinitely many measures, which are not absolutely continuous with respect to each other. Let us also mention the recent statistical advances of e.g. \cite{dereich2013constructive,fournier2015rate} on the convergence speed of the Wasserstein distance between the true and empirical measure, as well as  (dual) reformulations of $u(\x)$, or more generally Wasserstein distributionally robust optimization problems, see e.g. \cite{pflugwozabal,mohajerin2018data,blanchet2019quantifying,gao2022distributionally} and the survey articles \cite{pflug2018review,rahimian2019distributionally}. \\

In this work we are interested in an agent's optimal investment decision when facing an \textit{increasingly large} amount of Knightian uncertainty. Our main contribution is that, under minimal assumptions, an optimal portfolio choice exists and converges to some portfolio that minimizes the distance between the constraints set $ D$ and the origin. \\

Our convergence result can be contrasted with the recent works \cite{bartl2020robust} and \cite{obloj2021distributionally}, where the authors compute first-order sensitivities of the value function $u(\x)$ for $k \to 0$, i.e. for infinitesimally small amounts of model uncertainty. It is important to point out, that in order to analyze the other end of the asymptotic regime---namely infinitesimally large amounts of model uncertainty---as we do in this paper, we actually cannot rely on the analytic theory derived in these works. Instead we follow a different strategy, which can be summarized through the following three-step procedure: first, we prove that in $u(\x)$ we can restrict to bounded weights (see Proposition \ref{Kbound} and Lemma \ref{lem:bounded}). Then we give a reformulation of $u(\x)$, which essentially reduces $u(k,w)$ to the following optimization problem 
\begin{align}\label{eq:step2}
u(k,w)=\inf_{X^\mu, \,Z^\mu} 
\E_{\mu} [ U(\x+\x\langle w,X^{\mu} \rangle - |\x||Z^\mu||w|)],
\end{align}
where $X^{\mu}$ has law $\P$ under $\mu$ and $Z^\mu$ is such that $\E_{\mu}[|Z^\mu|^p]^{1/p}\le k$ for some atomless measure $\mu$, see Lemma \ref{lem:2}. Finally, we analyze this problem when $k\to \infty$. It turns out that while the distribution of $X^\mu$ does not change, the radially symmetric part $|Z^\mu|$ will dominate in \eqref{eq:step2} and if $|w|-|\overline{w}|>0$ for some $\overline{w} \in  \Dc$, $u(k,w)-u(k,\overline{w})$ goes to $-\infty$, see Proposition \ref{prop:main}. It is thus sufficient to recall that any $\overline{w} \in \Dc$ is a vector of $ D$ minimizing the Euclidean distance between $ D$ and  the origin. 
Following this idea of proof enables us to keep the assumptions on $U$ and $\P$ minimal. 

Note that \eqref{eq:step2} also implies that the domain of $U$  is necessarily the whole real line and not only $(0,\infty)$.  When $U$ is defined on $\R$, the existence of an optimal strategy has been proved in a one period model in the PhD thesis of R. Blanchard, see \cite{RB}. The generalization to the multiperiod case is provided in \cite{CF23}. When the uncertainty is modelled by $ B_k(\P)$, the proof simplifies significantly and we are able to prove existence of optimizers under minimal conditions on $U.$ We also emphasize that we do not postulate any no-arbitrage condition on the market. In fact, the quasi-sure no arbitrage condition of \cite{bouchard2015arbitrage} for $ B_k(\P)$ is satisfied automatically (see Proposition \ref{prop:na}).\\ 

Probably the paper closest to the results presented here is \cite{pflug20121}, where the authors consider a robust maximisation problem for risk measures. Their main examples concern the Markowitz functional and the conditional value at risk for the constraints set $D=\{w\in \R^d: \langle w, \mathds{1}\rangle=1 \}.$ However, their setting does not easily translate to ours: indeed, \cite[Proposition 1]{pflug20121} essentially assumes that the map $\Q\mapsto \|U'(X)\|_{L^p(\Q)}$ is constant on $B_k(\P)$, which is clearly not satisfied for general utility functions of interest. Furthermore, results  similar to \cite{pflug20121} are proved in \cite{sass2021robust} for the power and logarithmic utility functions when there is drift uncertainty in a multivariate continuous-time Black-Scholes type financial market. However, to the best of our knowledge, the case of utility maximization with a general constraints set has not been addressed, even in a one-period setup. Note that the most used portfolio choice models in the financial industry are indeed one-period models, which are especially popular with pension funds or banks, see e.g. \cite{basel}.  \\

Lastly let us mention that, while an extension of our main result to multiple periods seems natural to expect, the methods developed in this article do not seem adequate to provide a direct route for this extension. Indeed, for this one would first need to replace the Wasserstein ball $B_k(\P)$ by a ball in the adapted Wasserstein distance, see e.g \cite{backhoff2020adapted} and the references therein. However, we have failed to derive a dynamic programming principle for $u(k,w)$ as the one-step associated problems do not carry our minimal assumptions, in particular monotonicity of $\x\mapsto u(\x)$. In consequence we leave the extension of Theorem \ref{thm:main} to multiple periods for future research.\\

The rest of the paper is organized as follows. Section \ref{sec:notation} summarizes the notation used in the paper. After describing how we model the financial market in Section \ref{sec:finance} and detail the no arbitrage condition, we give a preliminary discussion of the assumptions of Theorem \ref{thm:main} in Section \ref{sec:discussion}. We then show existence of optimizers in different settings in Section \ref{sec:existence}. These results enable us to state the proof of Theorem \ref{thm:main}, which can be found in Section \ref{sec:proof}. After giving proofs omitted in the main text in Appendix \ref{sec:proofs}, we conduct a more in depth analysis of the function $u(k,w)$ in Appendix \ref{sec:appendix}, where we detail sufficient criteria for the existence and uniqueness of
worst-case measures in the ball $B_k(\P)$.

\section{Notation} \label{sec:notation}

Let us quickly summarize the notation used throughout the article: $C$ denotes a positive constant (which might change from occurrence to occurrence), $\langle \cdot, \cdot \rangle$ denotes the scalar product, $|\cdot|$ the Euclidean distance and $X:\R^l\to \R^l$ is the identity mapping for all $l\ge 1$. 
Let $p\ge 1.$  We denote by $\mathcal{P}_p(\R^d)$ the set of probability measures $\Q$ on $\R^d$ satisfying $\E_{\Q}[|X|^p]<\infty$. Fix $\P \in \mathcal{P}_p(\R^d)$ and define $C_{\P}:=(\E_{\P}[|X|^p])^{\frac1p}.$ 
We define the ball $B_k(\P)$ around $\P$ of radius $k>0$ with respect to the Wasserstein metric
\begin{align}\label{eq:wass}
\mathcal{W}_p (\P, \Q):=\Big( \inf_{\pi \in \Pi(\P, \Q)} \E_\pi[|X-Y|^p] \Big)^{1/p},
\end{align}
where $\Pi(\P, \Q)$ is the set of probability measures on $\R^d\times \R^d$ such that $X_*\pi=\P$ and $Y_*\pi=\Q$. Here we denote by $X:\R^d\times \R^d \to \R^d$ and $Y:\R^d\times \R^d\to \R^d$ the canonical projections to the first and second coordinate respectively, by a slight abuse of notation (which is standard in the optimal transport literature). The expression $f(X,Y)_*\pi$ or $f_*\pi$ stands for the push-forward measure of $\pi$ under the function $(x,y)\mapsto f(x,y)$, i.e. $f(X,Y)_*\pi(A)=f_*\pi(A):=\pi(f(X,Y)\in A)$ for all Borel sets $A$ of $\R^d\times \R^d$.  An application of the triangle inequality yields 
\begin{eqnarray}
\label{triangle}
\E_{\Q}[|X|^p]^{\frac1p}=\mathcal{W}_p(\Q, \delta_0) \leq \mathcal{W}_p(\P, \delta_0) + \mathcal{W}_p(\Q, \P)\leq C_{\P}+k<\infty
\end{eqnarray}
for all $\Q \in B_k(\P)$.

\section{Financial market}\label{sec:finance}
We consider a one period financial market with $d+1$ traded assets.
The returns $R^i$, $i \in \{0,\ldots,d\}$ represent the profit (or loss) created tomorrow from
investing one dollar's worth of asset $i$ today.
For notational simplicity we assume that  the num\'{e}raire asset has a constant price, i.e. $R^0=0$. 
An agent can trade in the market 
using    the portfolio weights $w=(w^0,\ldots,w^d)\in \R^{d+1}$ where $w^i$ is  the fraction of initial wealth $\x$ held in asset $i$, recalling that $\x\neq 0.$  So, 
$\sum_{i=0}^d w^i=1.$ 
The portfolio value at time one is given by 
\begin{align}\label{eq:V}
\begin{split}
V^{\x,w}&:=\x + \x\sum_{i=1}^{d} w^i R^i.
\end{split}
\end{align}

As stated in the Introduction, we model the distribution of $R:=(R^1, \dots, R^d)$ through a probability measure $\P$, i.e. the distribution of $R$ is equal to the distribution of the identity map $X$ under $\P$. By a slight abuse of notation we then write $V^{\x,w}= \x+\x\langle w, X\rangle$ for $w\in \R^d$. We can now state our definition of efficient markets relative to the prior $\P$:

\begin{definition}\label{def:na}
We say that the market has \emph{no arbitrage} $\mathrm{NA}(\P)$, if 
\begin{align*}
\forall \x\neq 0: \quad  V^{\x,w}\ge \x \quad \P\mathrm{-a.s.} \qquad \Rightarrow\qquad V^{\x,w} =\x \quad \P\mathrm{-a.s.}
\end{align*}
\end{definition}

A similar definition is used in \cite{RS06}[(3) and Proposition 1]. Definition \ref{def:na} is sometimes also called relative arbitrage with respect to the num\'{e}raire portfolio.
Using \eqref{eq:V} and $\x\neq 0$ we obtain that $\mathrm{NA}(\P)$ is equivalent to 
\begin{align*}
\langle  w, X \rangle \ge 0 \quad \P\text{-}\mathrm{a.s.} \qquad \Rightarrow\qquad \langle w, X\rangle=0 \quad \P\text{-}\mathrm{a.s.}
\end{align*}

We also define the following natural generalization to a set of priors $\mathcal{Q}$, going back to the seminal work of \cite{bouchard2015arbitrage}:
\begin{definition}
For a set $\mathcal{Q}$ of probability measures on $\R^d$ we say that quasi-sure no arbitrage NA($\mathcal{Q}$) holds, if for all $w\in \R^d$ 
\begin{align*}
\langle w, X\rangle\ge  0\quad \mathcal{Q}\mathrm{-q.s.} \qquad \Rightarrow\qquad \langle w, X\rangle =0 \quad \mathcal{Q}\mathrm{-q.s.}
\end{align*}
\end{definition}

The next proposition shows that the quasi-sure notion of no arbitrage is very weak when applied to Wasserstein balls $B_k(\P)$. Indeed, it holds independently if $\mathrm{NA}(\P)$ is true or not:
\begin{proposition}\label{prop:na}
$\mathrm{NA}(B_k(\P))$ holds for any $\P\in \mathcal{P}_p(\R^d)$ and any $k>0$.
\end{proposition}

\begin{proof}
Take $w\in \R^d$ such that $\langle w,X\rangle \ge 0$ $B_k(\P)\mathrm{-q.s.}$ 
Let $x\in \R^d.$ Lemma \ref{lem:1} below states that there exists $\alpha>0$ such that $\P_{x}= (1-\alpha)\P+\alpha \delta_{x}\in B_k(\P).$ 
Thus, as 
$
\langle w,X\rangle \ge 0$ $\P_{x}\text{-a.s.}$ 
holds,  we have that $\langle w, x\rangle\ge 0$. As $x\in \R^d$ was arbitrary, 
$w=0$ follows, showing the claim.
\end{proof}
We have used the following lemma:
\begin{lemma}\label{lem:1}
For any $\tilde{x}\in \R^d$  and $\alpha\in \left[0, \left(\frac{k}{C_{\P} +|\tilde{x}|}\right)^p\right)$ 
\begin{align*}
\P_{\tilde{x}}:= (1-\alpha)\P+\alpha \delta_{\tilde{x}}\in B_k(\P).
\end{align*}
\end{lemma}

\begin{proof} Let $\alpha>0$ and  $f(x,y):=(x,x)$. Choosing the coupling 
$\pi:= (1-\alpha) f_*\P+ \alpha (\P\otimes \delta_{\tilde{x}})\in \Pi(\P, \P_{\tilde{x}})$ we see that 
\begin{align*}
\mathcal{W}_p(\P_{\tilde{x}},\P)\le \left(\alpha \E_{\P}[|X-\tilde{x}|^p]\right)^{1/p}\le \alpha^{1/p} ((\E_{\P}[|X|^p])^{1/p}+|\tilde{x}|)= \alpha^{1/p} (C_{\P}+|\tilde{x}|).
\end{align*}
Setting the last expression less or equal to $k$ and solving for $\alpha$ proves the claim.  
\end{proof}

\begin{definition}
For $\Q \in \mathcal{P}_p(\R^d)$ and $\mathcal{Q}\subseteq \mathcal{P}_p(\R^d)$ we define the support of $\Q$ and the quasi-sure support of $\mathcal{Q}$ 
\begin{align*}
\mathrm{supp}(\Q)&:=\bigcap  \left\{ A \subset \mathbb{R}^{d},\; \mathrm{closed}, \;\Q(A) =1\right\}, \\
\mathrm{supp}(\mathcal{Q}) &:=\bigcap\left\{ A \subset \mathbb{R}^{d},\; \mathrm{closed},\; \Q(A)=1 \; \forall \mathbb{Q} \in \mathcal{Q} \right\}.
\end{align*}
We denote the affine hull of a set $A\subseteq \R^d$ by $\mathrm{Aff}(A)$.
\end{definition}

\begin{lemma}
We have $\mathrm{supp}(B_k(\P))= \R^d.$ 
\end{lemma}

\begin{proof}
For any $x \in \R,$ taking the $\P_{x}\in B_k(\P)$ from Lemma \ref{lem:1}  we see that 
\begin{align*}
\mathrm{supp}(B_k(\P)) \supseteq \bigcup_{x\in \R^d}\text{supp}(\P_{x}) \supseteq \bigcup_{x\in \R^d} \{x\} =\R^d.
\end{align*}

\end{proof}

The next proposition asserts that there exists a measure $\P^\ast$ in the Wasserstein ball satisfying the no arbitrage condition  and having full affine hull support. The measure $\P^\ast$ will be a key ingredient in the proof of Theorem \ref{thm:main}. 

\begin{proposition}\label{prop:past}
There exists $\P^\ast\in B_k(\P)$ such that $\mathrm{Aff}(\mathrm{supp}(\P^\ast))=\R^d$ and NA$(\P^\ast)$ holds.
\end{proposition}
Proposition \ref{prop:past} is proved  in a one-period setup in \cite{Bay17}[Lemma 2.2] and in the multi-period one in \cite{BC19}[Theorem 3.29]. The proof in our context is very simple and is given for sake of completeness. 
\begin{proof}
Denote the $i$th unit vector by $e_i=(0, \dots,0,1,0,\dots,0)$ and consider the probability measure 
\begin{align*}
\P^\ast:=\frac{1}{2d} (\P_{-e_1}+\P_{e_1}+\P_{-e_2}+\dots +\P_{e_d}),
\end{align*}
where $\P_{\pm e_i}\in B_k(\P)$ is defined as in Lemma \ref{lem:1} for $\tilde{x}=\pm e_i$ and $\alpha\in \left(0, \left(\frac{k}{C_{\P} +1}\right)^p\right).$ 
By convexity of the Wasserstein distance (see e.g. \cite[Theorem 4.8]{villani2008optimal}) and $\P_{\pm e_i}\in B_k(\P)$ we conclude that $\P^\ast\in B_k(\P)$. Furthermore 
\begin{align}\label{eq:ei}
\text{supp}(\P^\ast)\supseteq \{\pm e_i: i=1, \dots, d\},
\end{align}
so that $\mathrm{Aff}(\mathrm{supp}(\P^\ast))=\R^d.$ As \eqref{eq:ei} also implies that $0$ belongs to the (relative) interior of the convex hull of $\text{supp}(\P^\ast)$,  
we conclude that NA$(\P^\ast)$ holds true, see \cite{JS98}[Theorem 3g)].
\end{proof}

\begin{lemma}
\label{Rio}
Let $\mathbb{P}^* $ be as in Proposition \ref{prop:past}. There exists $\beta^\ast\in (0,1]$ such that for any $w \in \mathbb{R}^d \setminus \{0\}$,
\begin{align}
\label{alphaone}
\mathbb{P}^* \left( \langle w ,X\rangle  < -\beta^\ast |w|\right) \geq  \beta^\ast.
\end{align}
\end{lemma}
Condition \ref{alphaone} is called the quantitative no arbitrage condition and was proved in \cite{BC19}[Proposition 3.35]. Again, the proof is easy in our context and is reported for sake of completeness. 
\begin{proof}
We introduce for $n\geq 1$
$$
A_{n}:=\left\{ w \in  \R^d:\; |w|=1,\ \mathbb{P}^* \left(\langle w ,X\rangle < -\frac1n\right) < \frac1n  \right\}
$$
 and set $n_{0}:=\inf\{n \geq 1:A_{n}=\emptyset\}$ 
with the convention that $\inf \emptyset=+\infty$.  We prove by
contradiction that  $n_{0}<\infty$.
Assume that $n_{0}=\infty$ and choose $w_{n} \in A_n$ for all $n\ge 1$.
By passing to a subsequence we can assume that $w_{n}$ tends to  some $w^{*}\in \R^d$ with $ |w^* |=1$. 
Then   $\{\langle w^{*} ,X \rangle <0 \} \subseteq \liminf_{n} \left\{ \langle w_n ,X\rangle <   -{1}/{n} \right\}$ and Fatou's Lemma implies that $\mathbb{P}^* \left(\langle w^{*} ,X\rangle <0\right)  \leq 0. $ 
So $\langle w^{*} ,X\rangle \geq 0$ $\mathbb{P}^* $-a.s. As NA$(\P^\ast)$ holds (see Proposition \ref{prop:past}), $\langle w^{*} ,X\rangle = 0$ $\mathbb{P}^* $-a.s. and  
 \eqref{eq:ei} implies that  $w^{*}=0$,
which contradicts $|w^*|=1$. Thus $n_{0}<\infty$ and we can set  $\beta^\ast={1}/{n_{0}} \in (0,1]$. By definition of $A_{n_{0}}$, \eqref{alphaone} holds true.
\end{proof}
\section{Discussion of Theorem \ref{thm:main} and its assumptions} \label{sec:discussion}
All proofs not stated in this section are given in Appendix \ref{sec:proofs}. 
The first lemma shows that $\lim_{k \to \infty} u(k,w) =-\infty
$  for all $w\in \R^d \setminus\{0\}$ and justifies why we study in Theorem  \ref{thm:main} 
$\lim_{k \to \infty}  \mathrm{arg\ max}_{w \in D} u(k,w)$ and not 
  $\mathrm{arg\ max}_{w \in D} \lim_{k \to \infty} u(k,w)$.
\begin{lemma}
\label{lem:-infty}
Assume that $\lim_{x\to -\infty} U(x)=-\infty$. Then 
$
\lim_{ k \to \infty}u(k,w)=-\infty
$ 
for all $w\in \R^d \setminus\{0\}$ satisfying $\E_{\P}[U(\x+\x\langle w,X\rangle)]<\infty$.
\end{lemma}
Let us  single out \eqref{mino} in the statement of Theorem \ref{thm:main}.
\begin{assumption}\label{ass:1}
There exists $\underline{x}>0$ and $C_1>0$ such that $U(-\underline{x})<0$ and 
\begin{align*}
U(x)\ge -C_1(1+|x|^p)  \qquad\text{ for all }x\le-\underline{x},
\end{align*}
\end{assumption}
The following lemma shows that Assumption \ref{ass:1} is necessary to guarantee that $u(k,w)$ is well-defined.
\begin{lemma}
\label{lem:ill}
Assume that $U$ is non-decreasing and $\lim_{x\to -\infty} U(x)=-\infty$. If Assumption \ref{ass:1} is not satisfied, then $u(k,w)=-\infty$ 
for all $w\in \R^d \setminus\{0\}$ satisfying $\E_{\P}[U(\x+\x\langle w,X\rangle)]<\infty$.
\end{lemma}

We continue our discussion with a few remarks on the assumptions on asymptotic elasticity.  From an economics point of view, asymptotic elasticity can be interpreted as the ratio of the marginal utility $U'(x)$ and average utility,  $\frac{U(x)}{x}.$  
In the classical continuous time case of a fixed probability measure $\P,$ a utility function with domain equal to $\R$ must have a Reasonable Asymptotic Elasticity (RAE) at $\pm \infty$, i.e.
\begin{align*}
\liminf_{x\to -\infty} \frac{x U'(x)}{U(x)}>1 \quad \text{and}\quad \limsup_{x\to +\infty} \frac{x U'(x)}{U(x)}<1,
\end{align*}
in order to guarantee that the utility maximisation problem admits a solution, see \cite{Schach01}.
Having a RAE at $+ \infty$ means that the marginal utility should be smaller than the average utility for large positive financial positions, while a RAE at $- \infty$ means that marginal utility should be larger than the average utility for large negative positions. 
In discrete time it is enough to postulate RAE at $+ \infty$ or $-\infty$, see \cite{RS05}. 
RAE at $+ \infty$ or $-\infty$ provides control on $U$ at $+ \infty$ or $-\infty$ and concavity provides the other, see  Lemma \ref{lem:bc2}. 
In the case of multiple priors it is known that these conditions are not sufficient, see \cite{BC16} and \cite{CF23}. 

We now state a number of implications of concavity and RAE. 
\begin{lemma}
\label{lem:bc2}
Assume that $U$ is concave, non-decreasing and differentiable and that there exist positive constants $\underline{x}$ and $\overline{x}$ such that 
$U(-\underline x)<0$ and  $U(\overline x)>0$. Then
for all $x\in \R$ and $\lambda\geq 1$
\begin{align}
\label{bajka--}
U(\lambda x) & \leq  \lambda( U(x) + C),
\end{align}
where $C:=U^-(-\underline x) + U^+(\overline x)>0$. 
\end{lemma}
\begin{proof} 
The proof is similar to the proof of \cite{CF23}[Lemma 5.1] and thus omitted. 
\end{proof}

\begin{lemma}\label{lem:schach}
Assume that $U$ is concave, non-decreasing and differentiable. 
Assume that $\mbox{AE}_{+\infty}(U)<1$ and $\lim_{x\to \infty} U(x) >0$. Then there exist constants $\overline{x}>0$ and $\gamma <1$ 
such that $U(\overline{x})>0$ and for all $x\geq \overline{x}$  and $\lambda\geq 1$
\begin{align}\label{bajka+}
U(\lambda x) & \leq  \lambda^{\gamma} U(x). 
\end{align}
Moreover,  $\mbox{AE}_{+\infty}(U)$ is the infimum of ${\gamma}$ such that \eqref{bajka+} holds true. \\
Assume that  $\mbox{AE}_{-\infty}(U)>1$.  Then there exist constants $\underline{x}>0$ and $\gamma\rq{} >1$ 
such that $U(-\underline{x})<0$ and for all $x\leq - \underline{x}$  and $\lambda\geq 1$
\begin{align}\label{bajka-}
U(\lambda x) & \leq  \lambda^{\gamma\rq{}} U(x).
\end{align}
Moreover,  $\mbox{AE}_{-\infty}(U)$ is the supremum of $\gamma\rq{}$ such that \eqref{bajka-} holds true. 
\end{lemma}

\begin{proof}
The first assertion is proved in  \cite{KS99}[Lemma 6.3]. The second one can be proved similarly, see also \cite{CF23}[Proposition 5.3]. 
\end{proof}

\begin{lemma}\label{lem:bc2}
Assume that $U$ is concave, non-decreasing and differentiable. Assume that either $\mbox{AE}_{+\infty}(U)<1$ and $\lim_{x\to \infty}U(x)>0$ or 
$\mbox{AE}_{-\infty}(U)>1$. \\
Then there exist some strictly positive constants $\underline{x},\, \overline{x}, \,\underline{\gamma},\, \overline{\gamma}$ satisfying $\underline{\gamma}\leq 1 \leq \overline{\gamma}$ and $\underline{\gamma}<\overline{\gamma}$ and such that 
$U(-\underline{x})<0$ and  $U(\overline x)>0$ and 
for all $x\in \R$ and $\lambda\ge 1$
\begin{align}\label{conc}
U(\lambda x) & \leq  \lambda^{\underline{\gamma}}  \left(U\left(x \right)+{C}\right) \quad \mbox{ and } \quad
U(\lambda x)  \leq  \lambda^{\overline{\gamma}}  \left(U\left(x \right)+{C}\right),
\end{align}
where $C:=U^-(-\underline x) + U^+(\overline x)>0$. 
\end{lemma}

\begin{corollary}\label{cor:bc1}
Assume that $U$ is non decreasing and satisfies one of the inequalities in \eqref{conc}. Set $\hat  \gamma=\underline{\gamma}$  if the left hand side of \eqref{conc} holds true and 
$\hat \gamma=\overline{\gamma}$  if the right hand side of \eqref{conc} holds true. Then there exists some $\hat C>0$ such that 
\begin{align}
\label{eq:acp}U(x)  \leq \hat C(|x|^{\hat  \gamma}+1) \quad \mbox{for all } x\in \R.
\end{align}
Assume that there exist  constants $\underline{x}>0$ and $\gamma\rq{} >1$ 
such that $U(-\underline{x})<0$ and for all $x\leq - \underline{x}$  and $\lambda\geq 1,$ 
$U(\lambda x)  \leq  \lambda^{\gamma\rq{}} U(x).$ 
Then there exists some $C\rq{}>0$ such that 
\begin{align}
\label{eq:aep-}
U(x)  \leq -C\rq{}|x|^{\gamma\rq{}} \mbox{for all }x\le -\underline{x}.
\end{align}
However, the converses do not hold.
\end{corollary}
\begin{proof}
We prove that \eqref{eq:acp} holds true. 
Let $x\geq \overline{x}$. Choosing $\lambda=x/\overline{x} \geq 1$ in \eqref{conc} 
$$U(x)\leq \left(\frac{x}{\overline{x}} \right)^{\hat \gamma} (U(\overline{x})+C)= |x|^{\hat \gamma} \frac{2U(\overline{x})+U^-(-\underline x)}{\overline{x}^{\hat \gamma}}.$$
If $x\leq \overline{x}$ then $U(x)\leq U(\overline{x}). $
Thus trivially, the first assertion holds true for $\hat C:= (2U(\overline{x})+U^-(-\underline x))/\overline{x}^{\hat \gamma}+U(\overline{x}).$\\
For the second assertion let $x\leq - \underline{x}$. Choosing $\lambda=-x/\underline{x} \geq 1$ in $U(\lambda x)  \leq  \lambda^{\gamma\rq{}} U(x)$ 
$$U(x)\leq \left(\frac{-x}{\underline{x}} \right)^{\gamma\rq{}} U(-\underline{x})= -|x|^{\gamma\rq{}} \frac{U^-(-\underline x)}{\underline{x}^{\gamma\rq{}}}$$
and we set $C\rq{}:={U^-(-\underline x)}/{\underline{x}^{\gamma\rq{}}}.$
\end{proof}

\begin{remark}
\label{rem2}
If  $\mbox{AE}_{-\infty}(U)>1$ and Assumption \ref{ass:1} holds true, i.e.   for all $x\leq - \underline{x},$   
$U(x)  \ge -C_1(1+|x|^p)$, Lemma \ref{lem:schach} and  \eqref{eq:aep-} imply that $\gamma' \leq p$ and taking {supremum} over $\gamma'$ we conclude that 
$$1<\mbox{AE}_{-\infty}(U)\le p.$$
Thus in case (ii) of Theorem \ref{thm:main}, we must have that $\mbox{AE}_{-\infty}(U)\in (1,p]$.
\end{remark}

\section{Existence of bounded optimal solutions of $u(\x)$} \label{sec:existence}

We are now in a position to show that under the assumptions made in Theorem \ref{thm:main}, we can restrict to weights $w$ in a compact set $K\subseteq \R^d$ and that $
w_k=\mathrm{arg\ max}_{w \in  D} u(k,w)
$ exists. 

\subsection{Existence of bounded optimal solutions under RAE}

The following lemma holds true without assuming that $U$ is concave.

\begin{proposition}\label{Kbound}
Assume that $U$ is non-decreasing and that there exist strictly positive constants $x^*,\, C,\, \underline{\gamma},\, \overline{\gamma}$ satisfying $\underline{\gamma}\leq 1 \leq \overline{\gamma}$ and $\underline{\gamma}<\overline{\gamma}$ and such that 
$U(-x^*)\leq -(1+ C)$ and 
for all $x\in \R$ and $\lambda\ge 1$
\begin{align}\label{concl}
U(\lambda x) & \leq  \lambda^{\underline{\gamma}}  \left(U\left(x \right)+C\right) \quad \mbox{ and } \quad 
U(\lambda x)  \leq  \lambda^{\overline{\gamma}}  \left(U\left(x \right)+C\right).
\end{align}
Finally assume that there exists some $w^\ast \in D$ such that 
$$\sup_{\mathbb{Q} \in B_{k}(\mathbb{P} )} \mathbb{E}_{\mathbb{Q}}\left[U^-(x_0+ x_0 \langle w^\ast, X \rangle)\right] <\infty.$$
\noindent 
Then there exists $K=K(x_0,w^\ast)>0$, such that $|w^\ast|\leq K$ and 
\begin{align}
\label{eq:okae-}
u(\x)=\sup_{ w\in D,\  | w |\leq K}u(k,w).
\end{align}
\end{proposition}

\begin{proof}
We show the following: there exists $K=K(x_0, w^\ast)>0$, such that $|w^\ast|\leq K$ and 
for all $w \in D $ with $| w |\geq K,$
\begin{align}
\label{petitjoli}
u(k,w)=\inf_{\mathbb{Q} \in B_{k}(\mathbb{P} )} \mathbb{E}_{\mathbb{Q}}\left[U(x_0+ x_0 \langle w, X \rangle)\right] \leq \inf_{\mathbb{Q} \in B_{k}(\mathbb{P} )} \mathbb{E}_{\mathbb{Q}}\left[U(x_0+ x_0 \langle w^\ast, X \rangle)\right]=u(k,w^\ast).
\end{align} 
The statement of the proposition then follows immediately. The proof of \eqref{petitjoli} follows ideas of \cite{RS05} in the monoprior case and \cite[Proposition 3.24]{BC16} in the robust case. Let $$\P^\ast=\frac{1}{2d} (\P_{-e_1}+\P_{e_1}+\P_{-e_2}+\dots +\P_{e_d}) \in B_k(\P)$$ be as in Proposition \ref{prop:past}, where we choose 
$\alpha:=\alpha_{\pm e_i}= (\frac{k}{2(C_{\P}+1)})^p$ 
 for each $\P_{\pm e_i}$ (recall the construction in Lemma \ref{lem:1}). We calculate
\begin{align}\label{duracell}\nonumber
L^*&:= \mathbb{E}_{\P^\ast}\left[U^+\left(1+| X | \right)\right] =\frac{1}{2d}\sum_{i=1}^d 2\left(\left(1- \alpha \right)\mathbb{E}_{\P}\left[U^+\left(1+| X | \right)\right]  + \alpha  U^+(2)\right)\\
& = \left(1- \alpha \right)\mathbb{E}_{\P}\left[U^+\left(1+| X | \right)\right]  + \alpha  U^+(2)<\infty.
\end{align}
Indeed,  using \eqref{eq:acp}
\begin{align}\label{eq:aec}
U^+\left(1+| X | \right) & \leq \hat C(1+ |1+X|^{\underline \gamma}) \mathds{1}_{\{U(1+| X |) \ge 0\}}  \le \hat C (2+ |X|^{ \underline \gamma})
\end{align}
and as $ \underline \gamma<1 \leq p$ and  $\E_{\P}[|X|^{ \underline \gamma}]\leq1+\E_{\P}[|X|^p]=1+C_{\P}^p<\infty.$ 

As  $\underline{\gamma}<\overline{\gamma}$ there exists some  $0<\eta<1$ such that  $\underline{\gamma}<\eta \overline{\gamma}.$ Fix some $ x_0 \in \mathbb{R}$ and take $w \in \R^d$ such that  
$$| w |\geq K_0 (x_0):=\max \left(1,\frac1{|x_0 |^{\frac{1}{\eta}}}, \frac1{\beta^*} \left(1 + \frac{x^*}{|x_0|} \right),
\left(\frac{1+x^*}{\beta^*}\right)^{{\frac{1}{1-\eta}}}
 \right).$$
Let $\underline \lambda:=|x_0| | w |$ and $\overline \lambda: = |x_0|| w |^{\eta}.$ Then, %
$\underline \lambda=|x_0| | w |\geq  \overline \lambda= |x_0| | w |^{\eta}\geq 1.$ 

We start with the estimation of $U^+.$ 
Using that $U$ is non-decreasing, applying the Cauchy-Schwarz inequality and then \eqref{concl}, we get that 
\begin{align*}
 & U(x_0+ x_0 \langle w, X \rangle)  \leq  U(|x_0|+ |x_0| |w||X| ) = U\left(\underline\lambda \left(\frac{|x_0|}{\underline \lambda}+ \frac{|x_0|}{\underline \lambda} |w||X | \right)\right) \\ 
 & \quad \le \underline \lambda ^{\underline \gamma} \left( U\left(\frac{|x_0|}{\underline \lambda}+ \frac{|x_0|}{\underline \lambda} | w|| X | \right)+ C \right)
  \leq |x_0 |^{\underline \gamma} | w |^{\underline \gamma} \left( U^+\left(1+| X | \right) + 
 C  \right), 
\end{align*}
as  $| w | \ge 1$. Thus 
\begin{align}\label{eq:laurence4}
\begin{split}
\mathbb{E}_{\P^\ast}\left[
\mathds{1}_{\{U(x_0+ x_0 \langle w, X \rangle) \geq 0\}}
U(x_0+ x_0 \langle w, X \rangle)
\right]
&\le |x_0 |^{\underline \gamma} | w |^{\underline \gamma} \left( L^* + 
  C  \right), 
\end{split}
\end{align}
For the estimation of $U^-$ we introduce the event
\begin{eqnarray}
\label{Bset}
B:=\left\{ \mbox{sign} (x_0)\langle w , X \rangle < -\beta^*\right| w |\},
\end{eqnarray}
where $\mbox{sign} (x)=+1$ if $x\geq 0$ and  $\mbox{sign} (x)=-1$ if $x<0.$ \\
Recall that  $\overline \lambda \ge 1$.  
We start with inequalities on $B$. Recall that $w\neq0$. 
\begin{align}
\label{eq:laurence1}
\begin{split}
\x+ \x \langle w, X \rangle &\leq \overline \lambda \left( \frac1{| w |^{\eta}}+ \frac{\mbox{sign} (x_0)}{| w |^{\eta}}\langle w , X \rangle \right) \leq \overline \lambda\left( \frac1{| w |^{\eta}}- \beta^* | w |^{1-\eta} \right).
\end{split}
\end{align} 
As $|w| \geq 1$ and $\eta>0$, $|w|^{\eta} \geq 1$ and 
$ \frac1{| w |^{\eta}}- \beta^* | w |^{1-\eta} \leq 1 -  \beta^* | w |^{1-\eta}.$ As we have assumed that $|w|\geq \left(\frac{1+x^*}{\beta^*}\right)^{{\frac{1}{1-\eta}}}$, 
we get that 
$$\frac1{| w |^{\eta}}- \beta^* | w |^{1-\eta} \leq -x^*.$$
So using  that $U$ is increasing, \eqref{eq:laurence1} and \eqref{concl}, we get that 
\begin{align*}
&U(x_0+ x_0 \langle w, X \rangle)\mathds{1}_{B}
 {\le}  U\left(\overline \lambda\left( \frac1{| w |^{\eta}}- \beta^* | w |^{1-\eta} \right)\right)\mathds{1}_{B}\\
  &  \quad \le  |x_0|^{\overline \gamma} | w |^{\eta \overline \gamma} \left(U\left(\frac1{| w |^{\eta}}- \beta^* | w |^{1-\eta}\right) + C \right)\mathds{1}_{B}  \le  |x_0|^{\overline \gamma} | w |^{\eta \overline \gamma} \left( U(-x^*)+ C \right)\mathds{1}_{B}. 
\end{align*}
\noindent In conclusion, recalling that $U(-x^*)\leq -(1+ C),$ we obtain that 
\begin{align}\label{eq:laurence6}
\begin{split}
  U(x_0+ x_0 \langle w, X \rangle)\mathds{1}_{B} & \le   - |x_0|^{\overline \gamma} | w |^{\eta \overline \gamma} \mathds{1}_{B}.
  \end{split}
\end{align}
Now, we show that  $B \subseteq \{\x+ \x \langle w, X \rangle \leq - x^*\}.$
Indeed, on $B$ we have that 
\begin{align*}
\x+ \x \langle w, X \rangle &\leq  |x_0| \left( 1+ \mbox{sign} (x_0)\langle w , X \rangle \right) \le |x_0| \left(1- \beta^* | w |\right) \le -x^*,
 \end{align*}
 as we have assumed that $|w| \geq \frac1{\beta^*} \left(1 + \frac{x^*}{|x_0|} \right)$. 
 This shows the claim. Clearly then $B\subseteq \{U(x_0+ x_0 \langle w, X \rangle)< 0\}$ follows, as $U$ is non decreasing and $U(-x^*)<0$. 
 So recalling \eqref{eq:laurence6} 
\begin{align*}
U(x_0+ x_0 \langle w, X \rangle)\mathds{1}_{\{U(x_0+ x_0 \langle w, X \rangle)< 0\}} & \le  - |x_0|^{\overline \gamma} | w |^{\eta \overline \gamma} \mathds{1}_{B}.
\end{align*}
Recalling \eqref{alphaone} and that $w\neq0$, we have $\P^\ast(B )>\beta^*$. Combining these with the estimate \eqref{eq:laurence4}, we have for $| w |\geq K_0(x_0)$
\begin{align*}
u(k,w) \leq \mathbb{E}_{\P^\ast}\left[U(x_0+ x_0 \langle w, X \rangle)\right] & \leq |x_0 |^{\underline \gamma} | w |^{\underline \gamma} \left( L^* + 
 C  \right)   -{\beta^*} |x_0|^{\overline \gamma} | w |^{\eta \overline \gamma}.
\end{align*}
Thus, \eqref{petitjoli} holds, if 
\begin{align*}
& |x_0 |^{\underline \gamma} | w |^{\underline \gamma} \left( L^* + 
  C  \right) -\frac{\beta^*}2|x_0|^{\overline \gamma} | w |^{\eta \overline \gamma}  \leq  0\\
 & -\frac{\beta^*}2|x_0|^{\overline \gamma} | w |^{\eta \overline \gamma}    \leq -\sup_{\mathbb{Q} \in B_{k}(\mathbb{P} )} \mathbb{E}_{\mathbb{Q}}\left[U^-(x_0+ x_0 \langle w^\ast, X \rangle)\right] \leq u(k,w^\ast).
\end{align*}
As $\eta \overline{\gamma} -\underline{\gamma}>0$, this holds true if $|w| \geq K_1(x_0, w^\ast)$, where 
\begin{align*}
K_1(x_0, w^\ast):=\max
\left(
\Bigg( \frac{2\left( L^* +   C \right)}{\beta^*|x_0|^{\overline \gamma-\underline \gamma}} \right)&^{\frac1{\eta \overline{\gamma} -\underline{\gamma}}},\\
&\left( \frac{2\sup_{\mathbb{Q} \in B_{k}(\mathbb{P} )} \mathbb{E}_{\mathbb{Q}}\left[U^-(x_0+ x_0 \langle w^\ast, X \rangle)\right] 
  }{\beta^*| x_0 |^{\overline\gamma}} \right)^{  \frac{1}{\eta \overline\gamma}}
\Bigg).
\end{align*}
We define $K(x_0,w^\ast):= \max\left(K_0(x_0),K_1(x_0, w^\ast),|w^\ast|\right).$ This shows \eqref{petitjoli} and concludes the proof.
\end{proof}

\begin{corollary}\label{cor:rae-ex}
In the setting of Proposition \ref{Kbound} we have the following: if $U$ is continuous, $
w_k=\mathrm{arg\ max}_{w \in  D} u(k,w)
$ exists, $w_k \in  D$  and $|w_k|\leq K.$ 
\end{corollary}
\begin{proof}
Take a sequence of maximizers $(w_n)_n$ of 
\begin{align*}
u(\x)= \sup_{w \in D,\, | w |\leq K}
\inf_{\mathbb{Q} \in B_{k}(\mathbb{P} )} \mathbb{E}_{\mathbb{Q}}\left[U(x_0+ x_0 \langle w, X \rangle)\right],
\end{align*}
see \eqref{eq:okae-}. 
Then, as $|w_n|\leq K$,  after taking a subsequence $\lim_{n\to \infty} w_n=\hat{w},$ $|\hat{w}|\leq K$ and   $ \hat{w}\in  D$, as $D$ is closed. 
Then, for any $\Q\in B_k(\P)$,
\begin{align*}
u(\x)
&= \limsup_{n\to \infty} \inf_{\Q\in B_k(\P)} \E_{\Q}[U(\x+\x\langle w_n, X\rangle)]\\
&\le \limsup_{n\to \infty}\E_{\Q}[U^+(\x+\x\langle w_n, X\rangle)] -\liminf_{n\to \infty}\E_{\Q}[U^-(\x+\x\langle w_n, X\rangle)]\\
& \le\limsup_{n\to \infty}\E_{\Q}[U^+(\x+\x\langle w_n, X\rangle)] -\E_{\Q}[U^-(\x+\x\langle \hat{w}, X\rangle)],
\end{align*}
applying the Fatou lemma for $U^-$ which is continuous. 
Now using \eqref{eq:acp}  and the Cauchy-Schwarz inequality  we get that 
\begin{align*}
U^+(\x+\x\langle w_n, X\rangle))  & = (U(\x+\x\langle w_n, X\rangle) \mathds{1}_{U(\x+\x\langle w_n, X\rangle) \ge 0}\\
& \leq  \hat C (|\x+\x\langle w_n, X\rangle|^{\underline{\gamma}} + 1) \mathds{1}_{U(\x+\x\langle w_n, X\rangle) \ge 0}\\
& \leq \hat C(2^{\underline{\gamma}-1} |\x|^{\underline{\gamma}}(1+ K^{\underline{\gamma}}|X|^{\underline{\gamma}}) +1).
\end{align*}
Recall that $\underline{\gamma}<1 \leq p$ and that $\E_{\Q}[|X|^p]<\infty$ for all $\Q\in B_k(\P),$ see  \eqref{triangle}. So, 
$\E_{\Q}[|X|^{\underline{\gamma}}]\leq1+\E_{\Q}[|X|^p]<\infty.$ The dominated convergence theorem together with continuity of $U^+$ shows that 
$$\limsup_{n\to \infty}\E_{\Q}[U^+(\x+\x\langle w_n, X\rangle)] =\E_{\Q}[U^+(\x+\x\langle \hat{w}, X\rangle)].$$
Thus, $u(\x)\leq \E_{\Q}[U(\x+\x\langle \hat{w}, X\rangle)].$ 
Taking the infimum over all  $\Q\in B_k(\P)$, we find that 
$u(\x) 
\le \inf_{\mathbb{Q} \in B_{k}(\mathbb{P} )} \E_{\Q}[U(\x+\x\langle \hat{w}, X\rangle)] \le u(x_0)
$
as $\hat{w} \in D$ and $\hat{w}$ is an optimizer for $u(x_0)=\max_{w \in  D} u(k,w)$.
\end{proof}

\subsection{Existence of optimal solutions when $U$ is bounded from above}
\begin{lemma}
\label{lem:bounded}
Assume that $U$ is non-decreasing, continuous and bounded from above, as well as $\lim_{x\to -\infty} U(x)=-\infty.$ 
We also assume that there exists some $w^\ast \in  D$ such that $u(k, w^\ast)>-\infty$.  
Then there exists $K=K(x_0,w^\ast)>0$, such that $|w^\ast| \leq K$ and 
\begin{align}
\label{eq:okb}
u(\x)=\sup_{w \in D,  | w |\leq K }u(k,w).
\end{align}
Furthermore, $
w_k=\mathrm{arg\ max}_{w \in  D} u(k,w)
$ exists, $w_k \in  D$ and $|w_k|\leq K.$ 
\end{lemma}

\begin{proof}
We follow the same ideas as in the proof of Proposition \ref{Kbound}. In particular we show that 
there exists $K_1=K_1(\x)>0$, such that 
for all $w \in D$ with   $| w |\geq K_1,$
\begin{align}
\label{eq:uptopb}
u(k,w)\leq \inf_{\mathbb{Q} \in B_{k}(\mathbb{P} )} \mathbb{E}_{\mathbb{Q}}\left[U(x_0+ x_0 \langle w^\ast, X \rangle)\right]=u(k,w^\ast)
\end{align}
and \eqref{eq:okb} follows immediately choosing $K=K(\x,w^\ast)=\max(K_1(\x),|w^\ast|)$. 
Take $\P^\ast\in B_k(\P)$ as in Proposition \ref{prop:past}. Then \eqref{alphaone} implies that 
\begin{align*}
\P^\ast(\langle w, X\rangle <0)>0\quad \text{for all }w\in \R^d\setminus \{0\}.
\end{align*}
To prove \eqref{eq:uptopb}, we argue by contradiction: assume that for all $n \in \N$  there exists  $w_n \in D$ with $| w_n |\geq n$ and 
$u(k,w_n) > u(k,w^\ast).$ 
Then $\lim_{n\to \infty}|w_n|= \infty$. After taking a subsequence $\lim_{n\to \infty} w_n/|w_n|=\hat{w}$, and $|\hat{w}|=1$. 
On the other hand, $U$ is bounded from above by a constant $C>0$, and so
\begin{align*}
 \limsup_{n\to \infty} u(k,w_n) 
&\le \limsup_{n\to \infty}\E_{\P^\ast}[U(\x+\x\langle w_n, X\rangle)]\\
&\le \limsup_{n\to \infty}\E_{\P^\ast}[U(\x+\x|w_n| \langle \frac{w_n}{|w_n|}, X\rangle)\mathbf{1}_{\{\x\langle \hat{w}, X\rangle {<}0\}}] +C.
\end{align*}
Applying the reverse Fatou lemma to the last expression above and recalling that $U$ is continuous and $\lim_{x\to -\infty} U(x)=-\infty$ shows that 
$
 \limsup_{n\to \infty} u(k,w_n) =-\infty,
$
a contradiction to $ \limsup_{n\to \infty} u(k,w_n) \geq u(k,w^\ast)>-\infty$. \\
We now prove that $w_k$ exists. 
Take a sequence of maximizers $(w_n)_n$ of 
\begin{align*}
u(\x)= \sup_{w\in D,\ | w |\leq K }
\inf_{\mathbb{Q} \in B_{k}(\mathbb{P} )} \mathbb{E}_{\mathbb{Q}}\left[U(x_0+ x_0 \langle w, X \rangle)\right].
\end{align*}
Then, as $|w_n|\leq K$,  after taking a subsequence $\lim_{n\to \infty} w_n=\hat{w}$, where $|\hat{w}|\leq K$ and   $\hat{w} \in  D$, as $D$ is closed.
Then, for any $\Q\in B_k(\P)$,
\begin{align*}
u(\x)
&= \limsup_{n\to \infty} \inf_{\Q\in B_k(\P)} \E_{\Q}[U(\x+\x\langle w_n, X\rangle)]\\
&\le \limsup_{n\to \infty}\E_{\Q}[U(\x+\x\langle w_n, X\rangle)] \le \E_{\Q}[U(\x+\x\langle \hat{w}, X\rangle)],
\end{align*}
applying the reverse Fatou lemma and using that $U$ is continuous as before.
So taking the infimum over all  $\Q\in B_k(\P)$, we find that  
$u(\x)  \le \inf_{\mathbb{Q} \in B_{k}(\mathbb{P} )} \E_{\Q}[U(\x+\x\langle \hat{w}, X\rangle)] \le u(x_0)$ 
as $\hat{w} \in D$ and $\hat{w}$ is an optimizer for $u(x_0)=\max_{w \in  D} u(k,w)$.\\
\end{proof}

\section{Proof of Theorem \ref{thm:main}} \label{sec:proof}

To provide better intuition, it is helpful to rewrite the optimization problem $u(k, w)$ in a different way. 

\begin{lemma}\label{lem:villani}
Fix a Polish probability space $(\Omega, \mathcal{F},\mu)$ without atom. Then 
\begin{align}
\label{eq:Z}
u(k, w)= \inf_{X^\mu\sim \P, \, Z^\mu \in B^\mu_k } \E_{\mu} [ U(\x+\x\langle w,X^\mu+ Z^\mu\rangle)],
\end{align}
where $X^\mu:\Omega\to \R^d$ is a measurable map, the notation $X^\mu\sim \P$ means that $X^\mu$ has law $\P$ under $\mu$ and $ B^\mu_k$ is the set of measurable 
maps $Z^\mu:\Omega\to \R^d$ such that 
\begin{align*}
\|Z^\mu\|_{L_{\mu}^p(\R^d)} := (\E_\mu[|Z^\mu|^p])^{1/p}  \leq k.
\end{align*}
\end{lemma}

\begin{proof}
We only show the ``$\ge$"-inequality; the ``$\le$"-inequality follows very similarly.
We fix $\epsilon>0.$ Then there exists $\Q_{\epsilon}\in B_k(\P)$  such that 
$$u(k, w)=\inf_{\Q\in B_k(\P)} \E_\Q[U(\x+\x\langle w, X\rangle)] \geq \E_{\Q_{\epsilon}}[U(\x+\x\langle w, X\rangle)]-\epsilon.$$
There exists some  coupling $\pi_{\epsilon} \in \Pi(\P, \Q_{\epsilon})$  
satisfying $(\E_{\pi_{\epsilon}}[|Y-X|^p])^{1/p}\le k$ and thus
\begin{align*}
u(k, w)&\ge  \E_{\pi_{\epsilon}}[U(\x+\x\langle w, Y\rangle)]-\epsilon= \E_{\pi_{\epsilon}}[U(\x+\x\langle w, X+(Y-X)\rangle)]-\epsilon.
\end{align*}
Applying the [Villani, Ch.1, the measurable isomorphism p.7]\footnote{As we have not assumed that $(\R^d\times \R^d,(X, Y-X)_*\pi_{\epsilon})$ is atomless, the map $T_{\epsilon}$ is not necessary an isomorphism.}, we can find a measurable map $T_{\epsilon}: (\Omega, \mu) \to (\R^d\times \R^d, (X, Y-X)_*\pi_{\epsilon})$ such that ${T_{\epsilon}}_*\mu=(X, Y-X)_*\pi_{\epsilon}$. By a slight abuse of notation we call its first component $X^\mu_{\epsilon}$ and its second component $Z^\mu_{\epsilon}$. By definition, $X^\mu_{\epsilon}$ has law $\P$ under $\mu$ and 
$ (\E_\mu[|Z^\mu_{\epsilon}|^p])^{1/p}=(\E_{\pi_{\epsilon}}[|Y-X|^p])^{1/p}\le k.$  We then write
\begin{align*}
u(k, w)&\ge \E_{\pi_{\epsilon}}[U(\x+\x\langle w, X+(Y-X)\rangle)]-\epsilon = \E_\mu[U(\x+\x\langle w, X_{\epsilon}^\mu+Z_{\epsilon}^\mu \rangle)]-\epsilon \\
&\ge \inf_{X^\mu\sim \P,  \, Z^\mu \in B^\mu_k } \E_{\mu} [ U(\x+\x\langle w,X^\mu+ Z^\mu \rangle)]-\epsilon.
\end{align*}
This concludes the proof.
\end{proof}

The following observation is crucial:
\begin{lemma}\label{lem:2}
Fix $w\neq 0$ and $\epsilon>0$. Assume that $U$ is non-decreasing. Then there exists an $\epsilon$-optimizer $(X^\epsilon, Z^\epsilon)$ of \eqref{eq:Z} such that $X^\epsilon\sim \P$ under $\mu,$ $Z^\epsilon \in B^\mu_k$ and
\begin{align}
\label{eq:zepsi}
Z^\epsilon=-\mathrm{sign}(\x)|Z^\epsilon|  \frac{w}{|w|}. 
\end{align}
In particular
\begin{align}\label{eq:absolute}
\begin{split}
u(k,w)&=\inf_{X^\mu\sim \P, \, Z^\mu \in B^\mu_k } \E_{\mu} [ U(\x+\x\langle w,X^\mu+ Z^\mu\rangle)] \\
&= \inf_{X^\mu\sim \P, \, Z^\mu \in B^\mu_k } 
\E_{\mu} [ U(\x+\x\langle w,X^\mu\rangle - |\x||Z^\mu||w|)].
\end{split}
\end{align}
\end{lemma}

\begin{proof}
Take any $\epsilon$-optimizer $( X^\epsilon,\tilde{Z}^\epsilon)$ of \eqref{eq:Z}, i.e. $X^\epsilon \sim \P$ under $\mu$, $\tilde{Z}^\epsilon \in B^\mu_k$ and 
\begin{align*}
\E_{\mu} [ U(\x+\x\langle w, X^\epsilon+ \tilde{Z}^\epsilon \rangle)]\le \inf_{X^\mu\sim \P, \, Z^\mu \in B^\mu_k } \E_{\mu} [ U(\x+\x\langle w,X^\mu+ Z^\mu \rangle)]+\epsilon.
\end{align*}
Define 
\begin{align}\label{eq:zstar}
Z^\epsilon:=-\text{sign}(\x)|\tilde{Z}^\epsilon| \frac{w}{|w|}.
\end{align}
Then trivially $|Z^\epsilon|=|\tilde{Z}^\epsilon|$ and \eqref{eq:zepsi} holds true. Moreover,  
\begin{align}
\nonumber
&\inf_{X^\mu\sim \P, \, Z^\mu \in B^\mu_k } \E_{\mu} [ U(\x+\x\langle w,X^\mu+ Z^\mu \rangle)]+\epsilon \ge  \E_\mu[U(\x+\x\langle w, X^\epsilon+\tilde{Z}^\epsilon\rangle)]\nonumber\\
\label{eq:disp}
&  \stackrel{\text{(CS)}}{\ge} \E_\mu[U(\x+\x\langle w,  X^\epsilon \rangle-|\x||w||\tilde{Z}^\epsilon|)]
=  \E_\mu[U(\x+\x\langle w,X^\epsilon+ Z^\epsilon \rangle)],
\end{align}
as $\langle w, Z^\epsilon \rangle=-\text{sign}(\x)|\tilde{Z}^\epsilon| |w|$. 
This shows that $( X^\epsilon, Z^\epsilon)$ is also an $\epsilon$-minimizer. Next we show \eqref{eq:absolute}. 
As $X^\epsilon \sim \P$ and $\tilde{Z}^\epsilon \in B^\mu_k$ using \eqref{eq:disp} and letting $\epsilon\to 0$ show
\begin{align*}
&\inf_{X^\mu\sim \P, \, Z^\mu \in B^\mu_k } \E_{\mu} [ U(\x+\x\langle w,X^\mu+ Z^\mu\rangle)]\\
&\qquad\ge \inf_{X^\mu\sim \P, \, Z^\mu \in B^\mu_k } \E_\mu[U(\x+\x\langle w, X^\mu\rangle-|\x||w||Z^\mu| \rangle)].
\end{align*}
Next we show the reverse inequality 
by the same argument as above: we choose an $\epsilon$-optimizer $(X^\epsilon,{Z}^\epsilon)$ for 
the second infimum in \eqref{eq:absolute}  and set $Z^\ast:=-\text{sign}(\x)|{Z}^\epsilon| \frac{w}{|w|}.$ 
This yields
\begin{align*}
&\inf_{X^\mu\sim \P, \, Z^\mu \in B^\mu_k } \E_\mu[U(\x+\x\langle w, X^\mu\rangle-|\x||w||Z^\mu | \rangle)] +\epsilon\\
&\qquad\ge \E_\mu[U(\x+\x\langle w,  X^\epsilon \rangle-|\x||w||Z^\epsilon| \rangle)]= \E_\mu[U(\x+\x\langle w, X^\epsilon+Z^\ast \rangle)]\\
&\qquad\ge \inf_{X^\mu\sim \P, \, Z^\mu \in B^\mu_k } \E_{\mu} [ U(\x+\x\langle w, X^\mu+Z^\mu \rangle )].
\end{align*}
Taking $\epsilon\to 0$ concludes the proof.
\end{proof}
The following proposition is central for the proof of Theorem \ref{thm:main}.
\begin{proposition}\label{prop:main}
Let $U$ be concave, non-decreasing, differentiable, non-constant and let Assumption \ref{ass:1} hold.
Fix a compact subset $K\subseteq  D$ and $\hat{w} \in D$. Then the following holds: for any $\delta>0$ there exists $k_0=k_0(x_0, \delta, K,\hat{w}) \in \N$ such that for all $w\in K$ satisfying $|w|-|\hat{w}|\ge \delta$  and for all $k\ge k_0$ we have 
\begin{align}
\label{strictineq}
u(k,w)&=\inf_{\Q\in B_k(\P)} E_\Q[U(\x+\x\langle w, X\rangle)]< \inf_{\Q\in B_k(\P)} E_\Q[U(\x+\x\langle \hat{w}, X\rangle)]=u(k,\hat{w}).
\end{align}
\end{proposition}
\begin{proof}
Assume first that $\hat{w}\neq 0$.
By Lemma \ref{lem:2} we can rewrite the claim as 
\begin{align*}
&\inf_{X^\mu\sim \P, \, Z^\mu \in B^\mu_k } \E_{\mu} [ U(\x+\x\langle w,X^\mu\rangle - |\x||Z^\mu||w|)]= u(k,w)\\
&< u(k,\hat{w})  
=\inf_{X^\mu\sim \P, \, Z^\mu \in B^\mu_k } \E_{\mu} [ U(\x+\x\langle \hat{w},X^\mu\rangle -|\x||Z^\mu||\hat{w}|)].
\end{align*}
Let $(\hat{X}(k),\hat{Z}(k))$ be an $\epsilon$-optimizer for $u(k,\hat{w})$ given by Lemma \ref{lem:2}, so that by definition
\begin{align*}
&\inf_{X^\mu\sim \P, \, Z^\mu \in B^\mu_k } \E_{\mu} [ U(\x+\x\langle w,X^\mu\rangle - |\x||Z^\mu||w|)]\\
&\qquad- \inf_{X^\mu\sim \P, \, Z^\mu \in B^\mu_k } \E_{\mu} [ U(\x+\x\langle \hat{w},X^\mu\rangle  -|\x||Z^\mu||\hat{w}|)]\\
&\le \E_{\mu} [ U(\x+\x\langle w,\hat{X}(k)\rangle -|\x| |\hat{Z}(k)||w|)- U(\x+\x\langle \hat{w},\hat{X}(k)\rangle - |\x||\hat{Z}(k)||\hat{w}|)]+\epsilon.
\end{align*}
To simplify notation we set 
\begin{align*}
A_k(w)&:=\x+\x\langle w,\hat{X}(k)\rangle -|\x| |\hat{Z}(k)||w|.
\end{align*} 
Fix $\delta>0$. 
Using this notation, the claim follows if we show that there exists $k_0\in \N$ such that for all $w\in K$ satisfying $|w|-|\hat{w}|\ge \delta$  and for all $k\ge k_0$ we have 
\begin{align}\label{eq:aim}
\E_{\mu} [ U(A_k(w))- U(A_k(\hat{w}))]< -\epsilon.
\end{align}
Let $$G_k:=\x+\x\langle \hat{w},\hat{X}(k) \rangle+\x\langle|\hat{w}| \frac{\hat{w}-w}{\delta}, \hat{X}(k)\rangle.$$ 
By concavity of $U$
\begin{align}\label{eq:ableitung}
\begin{split}
U(A_k(w))- U(A_k(\hat{w})) & \leq (A_k(w)-A_k(\hat{w}))U\rq{}(A_k(\hat{w})) \\
U(A_k(\hat{w}))- U(G_k)   & \geq (A_k(\hat{w})-G_k)U\rq{}(A_k(\hat{w})).
\end{split}
\end{align}
Recall that $|w|-|\hat{w}|\ge \delta$ by assumption, and thus
\begin{align}\label{eq:ableitung2}
\begin{split}
A_k(\hat{w})-G_k & =  - |\x||\hat{Z}(k)||\hat{w}|+\x\langle|\hat{w}| \frac{w-\hat{w}}{\delta}, \hat{X}(k)\rangle\\
& = \frac{|\hat{w}|}{\delta}(- |\x||\hat{Z}(k)|\delta+\x\langle {w-\hat{w}}, \hat{X}(k)\rangle)\\
A_k(w)-A_k(\hat{w}) &=  \x\langle w -\hat{w},\hat{X}(k)\rangle -|\x||\hat{Z}(k)|(|w|-|\hat{w}|) \\
& \le \x\langle w -\hat{w},\hat{X}(k)\rangle -|\x||\hat{Z}(k)|\delta = \frac{\delta}{|\hat{w}|}(A_k(\hat{w})-G_k ). 
\end{split}
\end{align}
As $U^{\prime}\ge 0$ it follows from \eqref{eq:ableitung} and \eqref{eq:ableitung2} that 
\begin{align*}
U(A_k(w))- U(A_k(\hat{w})) & \leq  \frac{\delta}{|\hat{w}|}(A_k(\hat{w})-G_k )U\rq{}(A_k(\hat{w}))  \leq \frac{\delta}{|\hat{w}|}(U(A_k(\hat{w}))- U(G_k)). 
\end{align*}
To show \eqref{eq:aim} it thus remains to argue that the expectation of the last expression becomes small for large $k\in \N$; in fact we will show that it diverges to $-\infty$. For this we first note that by definition of $(\hat{X}(k),\hat{Z}(k))$ and Lemma \ref{lem:2} we have
\begin{align*}
\E_{\mu} [ U(A_k(\hat{w}))] &\le \inf_{X^\mu\sim \P, \, Z^\mu \in B^\mu_k } \E_{\mu} [ U(\x+\x\langle \hat{w},X^\mu\rangle - |\x||Z^\mu||\hat{w}|)]+\epsilon
=u(k,\hat{w}) +\epsilon. 
\end{align*}
Let $\varphi(x):= x-k\,\text{sign}(\x)\tfrac{\hat{w}}{|\hat{w}|}$ and define $\Q^*:= \varphi_*\P.$ Using the coupling 
$
\pi = (x\mapsto (x, \varphi(x))_*\P\in \Pi(\P, \Q^*)$
yields the upper bound
$$
\mathcal{W}_p(\P,\Q^*)\le \E_{\P}\Big[\big|k\frac{\hat{w}}{|\hat{w}|}\big|^p\Big]^{1/p}=k
$$
and thus $\Q^*\in B_k(\P)$. In particular
\begin{align}
\nonumber
& u(k, \hat{w})  \le  \E_{\Q^*} [ U(\x+\x\langle \hat{w},X\rangle)]= \E_{\P} [ U(\x+ \x \langle \hat{w},X\rangle -k|\x||\hat{w}|)] \\
\label{last?}
 &  \le  U(\x+ \x \E_{\P} [\langle \hat{w},X\rangle] -k|\x||\hat{w}|)  \leq  U(\x+ |\x| |\hat{w}|\E_{\P} [|X|] -k|\x||\hat{w}|) 
\end{align}
using Jensen's and Cauchy-Schwarz inequality. Note that $\E_{\P}[|X|]\leq 1+ \E_{\P}[|X|^p]<\infty$, see \eqref{triangle} and $\lim_{x\to -\infty} U(x)= -\infty$, as $U$ is concave, non-decreasing and non-constant. 
 In conclusion, the last expression in \eqref{last?} goes to $-\infty$ when $k\to \infty$ and 
\begin{align}\label{eq:1}
\lim_{k\to \infty} \E_{\mu} [ U(A_k(\hat{w}))]=-\infty.
\end{align}
On the other hand, as we have assumed that $w \in K$ and that $U$ is non-decreasing, the Cauchy-Schwarz and triangle inequality, together with \eqref{mino} yield
\begin{align}\label{eq:2}
\begin{split}
& \E_{\mu} [ U( G_k)]\ge \E_{\mu} [ U(-|\x|-|\x| \big| \hat{w} +|\hat{w}| \frac{\hat{w}-w}{\delta}\big|| \hat{X}(k)|)]\\
& \quad {\ge} \E_{\mu} [ U(-|\x|-|\x| | \hat{w}| (1+\frac{|\hat{w}|+|w|}{\delta}) |\hat{X}(k)|)] {\ge} \E_{\mu} [U ( -|\x|-\tilde{C}|\hat{X}(k)|)]\\
&\quad {\ge} -C_1\left(1+ 2^{p-1} (|\x|^p + \tilde{C}^p \E_{\P}[|X|^p]\right)>-\infty
\end{split}
\end{align}
for a constant $\tilde{C}=\tilde{C}(x_0, \delta, K,\hat{w})>0$. Here the last inequality follows as $\hat{X}(k)\sim \P$  under $\mu$, i.e. 
$\E_{\mu}[|\hat{X}(k)|^p]=\E_{\P}[|X|^p]<\infty.$ Combining \eqref{eq:1}-\eqref{eq:2} shows that 
$
\lim_{k\to \infty} \E_{\mu}[ U(A_k(\hat{w}))- U(G_k)]=-\infty,
$
as claimed. \\
If $\hat{w}=0$, $u(k,\hat{w})=U(\x)$. Take any $w\in D\setminus\{0\}$. Then as in \eqref{last?} 
$\E_{\P}[U(\x+\x\langle w,X\rangle)] \leq  U(\x+ |\x| |{w}|\E_{\P} [|X|])<\infty$.
So, Lemma \ref{lem:-infty} shows that 
$
\lim_{ k \to \infty}u(k,w)=-\infty
$ 
and the claim follows. 

\end{proof}

\begin{corollary}
\label{coro:main}
Let $U$ be concave, non-decreasing, differentiable, non-constant and let Assumption \ref{ass:1} hold.
Let $K$ be a  compact subset satisfying $K\subseteq D$ as well as $\Dc \cap K\neq \emptyset$, and choose  $\overline{w} \in \Dc \cap K$.  Fix $\delta>0$ and let $k_0\in \N$  
as in Proposition \ref{prop:main}. If $w_k\in K$ $\forall k\ge k_0$,  we have  
\begin{align}\label{eq:estimate}
0 \leq |w_k|-|\overline{w}|< \delta\qquad\text{for all }k\ge k_0.
\end{align}
Thus $\lim_{k\to \infty}w_k\in \Dc$ exists.
\end{corollary}

\begin{proof}
Fix $\delta>0$ and take $k\ge k_0.$ Assume that  $|w_k|-|\overline{w}|\ge \delta$. Then, as $w_k \in K$,
\eqref{strictineq} implies $u(k,w_k) < u(k,\overline{w}),$  which contradicts the definition of $w_k$. Thus, $|w_k|-|\overline{w}|< \delta$ and \eqref{eq:estimate} is proved as for all $w \in D$, $|w|\geq |\overline{w}|$, see  \eqref{eq:wu1}. \\
As  $\delta>0$ was arbitrary, $\lim_{k\to \infty} |w_k|=|\overline{w}|.$
As $(w_k)_k \subseteq K$, there exists a subsequence of $(w_k)_k$ (which we still denote by $(w_k)_k$), converging to some $\widetilde{w} \in K$. By continuity of $w\mapsto |w|$ we have that 
$
|\widetilde{w}|=|\lim_{k\to \infty} w_k|=\lim_{k\to \infty} |w_k|=|\overline{w}|.$ 
So, $\widetilde{w} \in \Dc$ and the claim follows.
\end{proof}

\noindent {\bf Proof of Theorem \ref{thm:main}.}\\
We assume that $U: \R \to \R$ is non-decreasing, non-constant, differentiable, concave and satisfies  \eqref{mino}. First note that 
$U^-(x)\le C_1\rq{}(1+|x|^{p})$  for all $x\in \R$ with the constant $C_1\rq{}:=C_1 -U(-\underline{x})>0. $
Let $w \in D.$ 
 \begin{align*} 
 u(k, w) & 
 \ge -\sup_{\Q\in B_k(\P)} \E_\Q[U^-(\x+\x\langle w, X\rangle)]  \ge -C_1\rq{}\sup_{\Q\in B_k(\P)} \E_\Q[1+|\x+\x\langle w, X\rangle|^{p}]. 
 \end{align*}
Using the Cauchy-Schwarz inequality, we get that 
$$1+|\x+\x\langle w, X\rangle|^{p} \leq 1+ |\x|^p(1 + |w||X|)^p \leq 1+ |\x|^p2^{p-1}(1 + |w|^p|X|^p).$$
 So, using  \eqref{triangle} we obtain
 \begin{align} 
 \label{eq:finim}
 \begin{split}
 u(k, w) & \ge -\sup_{\Q\in B_k(\P)} \E_\Q[U^-(\x+\x\langle w, X\rangle)] \\
& \ge -C_1\rq{} -C_1\rq{} |\x|^p2^{p-1}(1 + |w|^p\sup_{\Q\in B_k(\P)} \E_\Q[|X|^p])  >-\infty.
\end{split}
 \end{align}
As $D$ is non-empty and closed, $\Dc \neq \emptyset$ and we can choose some $\overline w\in \Dc \subset D.$

Assume first that $U$ is bounded from above. We can apply Lemma \ref{lem:bounded} as $U$ is  non-decreasing, continuous, $\lim_{x\to -\infty} U(x)=-\infty$ and 
$ u(k, \overline w)>  -\infty,$ see \eqref{eq:finim}. 
Thus  $w_k$ exists, $|w_k|\leq K=K(\x,\overline w),$ where $|\overline w|\leq K$. So, Corollary \ref{coro:main} applies with the compact $\{w \in  D : |w| \leq K \}$.

Assume now that either $\mbox{AE}_{+\infty}(U)<1$  or 
$\mbox{AE}_{-\infty}(U)>1$, and that $U$ is not bounded from above. This implies  $\lim_{x\to\infty}U(x)>0$ and thus
 Lemma \ref{lem:bc2} applies. There   
exists some strictly positive constants $\underline{x},\, \overline{x}, \,\underline{\gamma},\, \overline{\gamma}$ satisfying $\underline{\gamma}\leq 1 \leq \overline{\gamma}$ and $\underline{\gamma}<\overline{\gamma}$ and such that 
$U(-\underline{x})<0$ and  $U(\overline x)>0$, \eqref{conc}  holds true 
with $C:=U^-(-\underline x) + U^+(\overline x)>0$. As $\lim_{x\to -\infty} U(x)=-\infty$,  there exists $\lambda^*\geq 1$ such that 
$U(- \lambda^* \underline{x})\leq -(1+ C)$. As 
$$\sup_{\mathbb{Q} \in B_{k}(\mathbb{P} )} \mathbb{E}_{\mathbb{Q}}\left[U^-(x_0+ x_0 \langle \overline w, X \rangle)\right] <\infty,$$
see \eqref{eq:finim}, we can apply 
Corollary \ref{cor:rae-ex} with $x^*=\lambda^* \underline{x}.$ Thus $w_k$ exists, $|w_k|\leq K$ and again Corollary 
\ref{coro:main} applies with the compact $\{w \in  D : |w| \leq K \}$.

\appendix
\section{Remaining proofs from Section \ref{sec:discussion}} %
\label{sec:proofs}

\noindent{\bf Proof of Lemma \ref{lem:-infty}.}\\ 
Let $w\in \R^d \setminus\{0\}$ such that  $\E_{\P}[U(\x+\x\langle w,X\rangle)]<\infty$. 
Let for all $k\in \N$, $x_k:=-k \mbox{sign} (x_0)\frac{w}{|w|^2}.$
By definition we have $\x+\x\langle w, x_k \rangle =\x -k|\x|$. 
We now define 
$$\alpha_{x_k}:=\frac{k^p}{2^{p-1}\left(\frac{k^p}{|w|^p} +C_{\P}^p\right)}\le \left(\frac{k}{C_{\P}+|x_k|}\right)^p, \quad \P_{x_k} := (1-\alpha_{x_k}) \P+\alpha_{x_k}\delta_{x_k}.$$ Then $\P_{x_k}\in B_k(\P)$ by Lemma \ref{lem:1} and we obtain
\begin{align*}
u(k,w) & \le
 \E_{\P_{x_k}} [U(\x+\x\langle w, X\rangle)] \\
 &\le  (1-\alpha_{x_k}) \E_{\P}[ U(\x+\x\langle w, X\rangle)]+ \alpha_{x_k} U(\x+\x\langle w, x_k \rangle)\\
& \le \max(0, \E_{\P}[ U(\x+\x\langle w, X\rangle)]) +\frac{k^p}{2^{p-1}\left(\frac{k^p}{|w|^p} +C_{\P}^p\right)} U(\x -k|\x|) 
\end{align*}
which tends to $ -\infty
$ for $k\to \infty$, showing the claim.\\
\mbox\\

\noindent{\bf Proof of Lemma \ref{lem:ill}.}\\ 
The proof is similar to the proof of Lemma \ref{lem:-infty}. 
Let $w\in \R^d  \setminus\{0\}$ such that   $\E_{\P}[U(\x+\x\langle w,X\rangle)]<\infty$. As $\lim_{x\to -\infty} U(x)=-\infty$, there exists some  $\underline{x}>0$ such that $U(-\underline{x})<0$ and 
if Assumption \ref{ass:1} is not satisfied, there exists a sequence $(y_n)_n$ such that $y_n \leq - \underline{x}$ and 
\begin{align}\label{eq:bound}
U(y_n) \leq - n(1+|y_n|^p)\qquad \forall n\in \N.
\end{align}
We see easily that $\inf_n y_n=-\infty$. 
Taking a subsequence if necessary we can thus assume that $(y_n)_n$ are negative, decreasing and $\lim_{n\to \infty} y_n=-\infty$. Next we define the sequence $(x_n)_n$ of vectors in $\R^d$ via 
\begin{align}
\label{bonchoix}
x_n=\frac{y_n}{|w|^2\x}\,w-\frac{w}{|w|^2},
\end{align} recalling that $w\neq 0$ and $\x\neq 0$. By definition we have $\x+\x\langle w, x_n \rangle =y_n$ for all $n\in \N$. We now define 
$$\alpha_{x_n}:=\frac{k^p}{2^{p-1}(|x_n|^p +C_{\P}^p)}\le \left(\frac{k}{C_{\P}+|x_n|}\right)^p, \qquad \P_{x_n} := (1-\alpha_{x_n}) \P+\alpha_{x_n}\delta_{x_n}.$$ Then $\P_{x_n}\in B_k(\P)$ by Lemma \ref{lem:1} and we obtain recalling \eqref{eq:bound}
\begin{align*}
u(k,w) &\le
 \E_{\P_{x_n}} [U(\x+\x\langle w, X\rangle)] \le  (1-\alpha_{x_n}) \E_{\P}[ U(\x+\x\langle w, X\rangle)]+ \alpha_{x_n} U( y_n)\\
&\le  \max(0, \E_{\P}[ U(\x+\x\langle w, X\rangle)]) -n \alpha_{x_n} (1+|y_n|^p).
\end{align*}
Now recalling \eqref{bonchoix}, $|x_n|^p\le \frac{2^{p-1}}{|w|^p} \max(\frac1{|\x|^p} ,1)(1+|y_n|^p)$  and we get that 
\begin{align*}
n \alpha_{x_n} (1+|y_n|^p) &  \ge 
n\frac{k^p(1+|y_n|^p)}{2^{p-1}\left(\frac{2^{p-1}}{|w|^p} \max\left(\frac1{|\x|^p} ,1\right))(|y_n|^p+1) +C_{\P}^p\right)}\to  +\infty
\end{align*}
when $n\to \infty.$ Thus, 
we obtain that $u(k,w)  =  -\infty,$  
 showing the claim.\\
\mbox\\

\noindent {\bf Proof of Lemma \ref{lem:bc2}.}\\
Assume that $\mbox{AE}_{+\infty}(U)<1$ and $\lim_{x\to \infty} U(x) >0$. Then \eqref{bajka+} holds true. 
Note that if $x\leq \overline{x}$  and $\lambda\geq 1$, as $U$ is non-decreasing,  \eqref{bajka+} implies that 
\begin{align}\label{bajka2}
U(\lambda x) & \leq  U(\lambda \overline{x}) \le \lambda^{\gamma} U(\overline{x}). 
\end{align}
We claim that \eqref{bajka+} and \eqref{bajka2}  imply the following: for all $x\in \R$ and $\lambda\geq 1$ 
\begin{align}
\label{bajka++}
U(\lambda x) & \leq  \lambda^{\gamma}( U(x) + C).
\end{align}
If $x \geq  \overline{x}$, \eqref{bajka++} is a direct consequence of \eqref{bajka+}. 
If  $-\underline{x}\le x\le \overline{x},$  as $C=U^-(-\underline x) + U^+(\overline x)>0,$
\eqref{bajka2} implies that 
\begin{align*}
U(\lambda x)\le  \lambda^{\gamma} U(\overline{x})=  \lambda^{\gamma} (U(-\underline{x})+C) \le  \lambda^{\gamma} (U(x)+C).
\end{align*}
If $x<-\underline{x}\le 0$, %
$U(\lambda x)<0$. We distinguish two cases: (i) $U(x)+C\ge 0$: in this case \eqref{bajka++} is trivial. (ii) $U(x)+C< 0$: recalling that $\lambda\ge \lambda^\gamma$ we conclude from \eqref{bajka--} that
\begin{align*}
U(\lambda x) & \leq \lambda (U(x) + C)\le \lambda^\gamma (U(x) + C).
\end{align*} 
This shows \eqref{bajka++}.
Recalling \eqref{bajka--}, we can choose $\underline{\gamma}:=\gamma <1=:\overline{\gamma}$ and 
\eqref{conc} is proved. 
Assume now that  $\mbox{AE}_{-\infty}(U)>1$. Then \eqref{bajka-} holds true and implies that  for all $x\leq - \underline{x}$  and $\lambda\geq 1$
\begin{align*}
U(\lambda x) & \leq  \lambda^{\gamma\rq{}} U(x)\le  \lambda^{\gamma\rq{}} (U(x)+C).
\end{align*}
Assume now that $x\geq 0$. Then, for all $\lambda\geq 1$, $-\underline x < 0 \leq x \leq \lambda x$. As $U$ is concave 
 \begin{align*}
\frac{U(\lambda x)-U(-\underline x)}{\lambda x+\underline x} & \leq  \frac{U(x)-U(-\underline x)}{ x+\underline x}.
\end{align*}
As $   \frac{\lambda x+\underline x}{ x+\underline x}\leq \lambda \leq \lambda^{\gamma\rq{}}$  and $U(-\underline x) \leq 0,$ 
 \begin{align*}
U(\lambda x) 
& \leq \lambda^{\gamma\rq{}} (U(x)+U^-(-\underline x))\le  \lambda^{\gamma\rq{}} (U(x)+C).
\end{align*}
Assume that $ - \underline{x} \leq x\leq 0.$ Then $\lambda x \leq 0 \leq \overline x$ and $U(\lambda x)
 \leq U(\overline x)=U^+(\overline x).$ 
As $U$ is non-decreasing  and $\lambda^{\gamma\rq{}} \geq 1$, we obtain that 
 \begin{align*}
\lambda^{\gamma\rq{}}( U( x) +C) 
& \geq \lambda^{\gamma\rq{}}( U( -\underline x) +C) =\lambda^{\gamma\rq{}} U^+(\overline x)\ge  U^+(\overline x).
\end{align*}
So we have proved that \eqref{bajka++} holds true with $\gamma\rq{}$ instead of $\gamma$. 
Recalling \eqref{bajka--}, we can choose $\underline{\gamma}:=1 <\gamma\rq{}=:\overline{\gamma}$ and 
\eqref{conc} is  proved. 
\section{Further results on $u(k,w)$ } \label{sec:appendix}
In this part of the appendix we dive deeper into the analysis of the max-min problem $u(k,w)$; in particular we are interested in understanding when the infimum are attained. It turns out that this is the case under some slightly stricter growth assumptions than those stated in Assumption \ref{ass:1} together with a constraint on the measure $\P$. 
\begin{assumption}\label{ass:2}
Assume that $p>1$ and that there exist $C_2>0$ and $1 \le r<p$ such that 
\begin{align*}
U^-(x)\le C_2(1+|x|^{r}) \qquad\text{ for all }x\in \R.
\end{align*}
\end{assumption}
\begin{lemma}\label{lem:ass1}
Assume that $U$ is non-decreasing and concave, Assumption \ref{ass:2} implies that  there exists $C_2\rq{}>0$ such that
\begin{align}\label{eq:growth}
|U(x)|\le C_2\rq{}(1+|x|^{r}) \qquad\text{ for all }x\in \R.
\end{align}
\end{lemma}
\begin{proof}
Let $y\in \R$ be the smallest number satisfying $U(y)>0$. By continuity of $U$ we have $U(y)=0$. If $z\leq y,$ the monotonicity of $U$ implies that $U(z)\leq U(y)=0$ and  \eqref{eq:growth} holds true. 
If $z >y$, setting $x=z-y>0,$ concavity and monotonicity of $U$ clearly imply that  
$$0=U(y)\ge\frac{1}{2} (U(y+x)+U(y-x))=\frac{1}{2} (U^+(y+x)-U^-(y-x)),$$
so that $U^+(y+x)\le U^-(y-x)$. In particular \eqref{eq:growth} thus implies
\begin{align*}
U^+(z) = U^+(y+x) &\le U^-(y-x) \le C_2 (1+|y-x|^r) \le  C_2'(1+|z|^{r}),
\end{align*}
writing $C_2':=C_2(1+ 2^{r-1}|2y|^r+2^{r-1})>C_2.$ 
Thus, for all $z\in \R,$ \eqref{eq:growth} implies that 
\begin{align*}
|U(z)|=U^+(z)+U^-(z) \le 2C_2'(1+|z|^{r}).
\end{align*}
\end{proof}

\begin{lemma}
\label{lemmex}
Fix $w\neq 0$. Under Assumption \ref{ass:2} there exists an optimizer $\Q^\ast$ of  \eqref{eq:prob}, i.e. 
\begin{align*}
u(k,w)= \inf_{\Q\in B_k(\P)} \E_\Q[U(\x+\x\langle w, X\rangle)]=\E_{\Q^\ast}[U(\x+\x\langle w, X\rangle)].
\end{align*}
Moreover, there exists an optimizer 
 $(X^\ast,Z^\ast)$ of \eqref{eq:Z}, i.e. 
 \begin{align*}
u(k, w)= \inf_{X^\mu\sim \P, \, Z^\mu \in B^\mu_k } \E_{\mu} [ U(\x+\x\langle w,X^\mu+ Z^\mu\rangle)]= \E_{\mu} [ U(\x+\x\langle w,X^\ast+ Z^\ast\rangle)],
\end{align*}
such that $X^*\sim \P$ under $\mu,$  $Z^* \in B^\mu_k $ and 
\begin{align}
\label{eq:normast}
Z^\ast=-\mathrm{sign}(x_0)|Z^\ast| \frac{w}{|w|}.
\end{align}
\end{lemma}
\begin{proof}
We first prove existence of optimizers for \eqref{eq:prob}. Take a sequence $(\Q_n)_n$ in $B_k(\P)$ of $1/n$-optimizers of $u(k,w)$, i.e. 
$
u(k,w)\le \E_{\Q_n}[U(\x+\x\langle w, X\rangle)]\le u(k,w)+{1}/{n}.
$ 
As $B_k(\P)$ is compact wrt. $W_r$ by Lemma \ref{lem:compact} below, there exists a convergent subsequence, still called $(\Q_n)_n$, converging to some $\Q^\ast \in B_k(\P)$ in $W_r$. By \eqref{eq:growth}  and Cauchy-Schwarz inequality 
\begin{align*}
|U(\x+\x\langle w, X\rangle)| &\le  C_2\rq{}(1+|\x+\x\langle w, X\rangle|^{r}) \\
&\le  C_2\rq{}(1+|\x|^r 2^{r-1}(1+| w|^r |X|^{r})) \le  C(1+|X|^{r})
\end{align*}
and [Villani, 2008, Def. 6.8 p108] shows that   
\begin{align*}
u(k,w)=\lim_{n\to \infty} \E_{\Q_n}[U(\x+\x\langle w, X\rangle)] =\E_{\Q^\ast} [U(\x+\x\langle w, X\rangle)].
\end{align*}
This concludes the proof of existence of optimizers for \eqref{eq:prob}. 
Using the same arguments as in the proof of Lemma \ref{lem:villani}  there exists $( X^\ast,\tilde{Z})$ such that $X^\ast\sim \P$ under $\mu,$  $\tilde{Z} \in B_k^{\mu}$ and
\begin{align}
\label{eq:opQ}
u(k,w)=\E_{\Q^\ast} [U(\x+\x\langle w, X\rangle)]=\E_\mu [U(\x+\x\langle w,  X^*+\tilde{Z}\rangle)].
\end{align}
Define $Z^\ast:=-\text{sign}(\x)|\tilde{Z}| \frac{w}{|w|}.$ 
Then trivially $|Z^\ast|=|\tilde{Z}|$ and \eqref{eq:normast} holds true. Using Lemma \ref{lem:villani} and \eqref{eq:opQ}, we get that 
\begin{align*}
u(k,w) &= \inf_{X^\mu\sim \P, \, Z^\mu \in B^\mu_k } \E_{\mu} [ U(\x+\x\langle w,X^\mu+ Z^\mu\rangle)] \\
&= \E_\mu[U(\x+\x\langle w, X^\ast+\tilde{Z}\rangle)]\\
&\stackrel{\text{(CS)}}{\ge} \E_\mu[U(\x+\x\langle w, X^\ast\rangle-|\x||w||\tilde{Z}|)]\\
&=  \E_\mu[U(\x+\x\langle w, X^\ast\rangle+\x\langle w, Z^\ast \rangle)]\\
& \ge \inf_{X^\mu\sim \P, \, Z^\mu \in B^\mu_k } \E_{\mu} [ U(\x+\x\langle w,X^\mu+ Z^\mu\rangle)].
\end{align*}
This shows that $(X^\ast,Z^\ast)$ is also a minimizer and concludes the proof.
\end{proof}

\begin{lemma}\label{lem:compact}
Let $1\le r<p$. Then the $p$-Wasserstein ball $B_k(\P)$ is compact wrt. $W_{r}$.
\end{lemma}

\begin{proof}
Recall that \eqref{triangle} shows that $\sup_{\Q\in B_k(\P)}\E_{\Q}[|X|^p]^{1/p}<\infty$.
Then an application of Prokhorov's theorem shows that $B_k(\P)$ is weakly precompact. Hence, for every sequence of measures $(\Q_n)_{n}$ in $B_k(\P)$ there exists a subsequence, which we also call $(\Q_n)_{n}$ and a measure $\Q$ such that $\Q_n$ converges weakly to $\Q$. As $W_p$ is weakly lower semicontinuous (see \cite[Rem. 6.12, p.109]{villani2008optimal}), this implies $\Q \in B_k(\P)$. It remains to show that $\Q_n$ converge to $\Q$ in $W_r$. By \cite[Definition 6.8 and Theorem 6.9]{villani2008optimal} it is sufficient to show that $\lim_{n\to \infty} \E_{\Q_n}[|X|^r]=\E_{\Q}[|X|^r]$. For this we note that
\begin{align*}
& \lim_{n\to \infty} \E_{\Q_n}[|X|^r] = \lim_{l \to \infty} \lim_{n\to \infty} \left(\E_{\Q_n}[|X|^r\mathds{1}_{\{|X|\le l\}}] + \E_{\Q_n}[|X|^r\mathds{1}_{\{|X|> l\}}]\right)\\
& \quad \le \limsup_{l\to \infty}  \limsup_{n\to \infty} \E_{\Q_n}[(|X|\wedge l)^r] + \limsup_{l \to \infty} \sup_{n\in \N}  \E_{\Q_n}[|X|^r\mathds{1}_{\{|X|> l\}}]\\
& \quad \le \limsup_{l\to \infty} \E_{\Q}[(|X|\wedge l)^r] + \limsup_{l \to \infty} \sup_{n\in \N}\Big(\E_{\Q_n}[|X|^p]\Big)^{r/p} \Q_n(|X|>l)^{(p-r)/p}\\
& \quad \le \E_{\Q}[|X|^r] +\limsup_{l \to \infty} \sup_{n\in \N}\Big(\E_{\Q_n}[|X|^p]\Big)^{r/p} \Bigg(\frac{\E_{\Q_n}[|X|^p]}{l^p}\Bigg)^{(p-r)/p}\\
& \quad \le \E_{\Q}[|X|^r] +\limsup_{l \to \infty}  \Bigg(\frac{\sup_{n\in \N}\E_{\Q_n}[|X|^p]}{l^{p-r}}\Bigg) \le \E_{\Q}[|X|^r] +\limsup_{l \to \infty}  \Bigg(\frac{(C_{\P}+k)^p}{l^{p-r}}\Bigg) \\
&\quad = \E_{\Q}[|X|^r],
\end{align*}
where we have used weak convergence and H\"older's inequality for the second inequality, the dominated convergence theorem and Markov's inequality for the third inequality and \eqref{triangle} and $r<p$ for the last one. 
Furthermore,
\begin{align*}
\lim_{n\to \infty} \E_{\Q_n}[|X|^r]&= \lim_{l \to \infty} \lim_{n\to \infty} (\E_{\Q_n}[|X|^r\mathds{1}_{\{|X|\le l\}}+|X|^r\mathds{1}_{\{|X|> l\}}])\\
&\ge \liminf_{l\to \infty}  \liminf_{n\to \infty} \E_{\Q_n}[(|X|\wedge l)^r]= \E_{\Q}[|X|^r].
\end{align*}
This concludes the proof.
\end{proof}

Under slightly stricter assumptions we can strenghten Lemma \ref{lem:2} in the following way:
\begin{lemma}\label{lem:Z}
Fix $w\neq 0$ and assume that $U$ is strictly increasing and satisfies Assumption \ref{ass:2}. Then any optimizer $(X^\ast,Z^\ast)$ of \eqref{eq:Z} satisfies
\begin{align*}
Z^\ast=-\mathrm{sign}(x_0)|Z^\ast| \frac{w}{|w|}.
\end{align*}
Furthermore $\|Z^*\|_{L_{\mu}^p(\R^d)}=k$ holds for any optimizer of \eqref{eq:Z}.
\end{lemma}

\begin{proof}
Take an optimizer $(X^\ast,Z^\ast)$ of \eqref{eq:Z}, which exists by Lemma  \ref{lemmex}, define 
\begin{align*}
A:=\Big\{Z^\ast\neq -\mathrm{sign}(x_0)|Z^\ast|\frac{w}{|w|}\Big\}
\end{align*}
and assume that $\mu(A)>0$. Recall from the Cauchy-Schwarz inequality that we have 
\begin{align}\label{eq:cs}
\x\langle w, z\rangle  >- |\x||w| |z|=\x\langle w, -\text{sign}(\x)|z| \frac{w}{|w|}\rangle
\end{align}
for all $z\in \R^d$ such that $z\neq - \text{sign}(\x)|z|w/|w|$. Similarly to the proof of Lemma \ref{lemmex} we have
\begin{align*}
&\inf_{X^\mu\sim \P, \, Z^\mu \in B^\mu_k } \E_{\mu} [ U(\x+\x\langle w,X^\mu+ Z^\mu\rangle)]= \E_\mu[U(\x+\x\langle w, X^\ast+Z^\ast\rangle)]\\
&\qquad> \E_\mu[U(\x+\x\langle w, X^\ast-\text{sign}(\x)|Z^\ast|\frac{w}{|w|}\rangle)\mathds{1}_A] + \E_\mu[U(\x+\x\langle w, X^\ast+Z^\ast\rangle)\mathds{1}_{A^c}]\\
&\qquad= \E_\mu[U(\x+\x\langle w, X^\ast-\text{sign}(\x)|Z^\ast|\frac{w}{|w|}\rangle)] \\
&\qquad \ge \inf_{X^\mu\sim \P, \, Z^\mu \in B^\mu_k } \E_{\mu} [ U(\x+\x\langle w,X^\mu+ Z^\mu \rangle)],
\end{align*}
where we used \eqref{eq:cs} and that $U$ is strictly increasing for the strict inequality and also that 
$\|-\text{sign}(\x)|Z^\ast|\frac{w}{|w|}\|_{L_{\mu}^p(\R^d)}=\|Z^\ast\|_{L_{\mu}^p(\R^d)} \le k.$ This leads to a contradiction to $\mu(A)>0$. 
Furthermore, from the fact that $U$ is strictly increasing it is clear that $0\neq Z^\ast$ for any optimizer $(X^\ast,Z^\ast)$. Indeed, otherwise for any $\delta>0$
\begin{align*}
u(k,w)&= \E_\mu[U(\x+\x\langle w, X^\ast\rangle)]\\
&>\E_\mu[U(\x+\x\langle w, X^\ast\rangle - |\x|\delta )]= \E_\mu[U(\x+\x\langle w, X^\ast \rangle - \x \langle w, \delta \frac{w}{|w|^2}\text{sign}(\x) \rangle)].
\end{align*}
As  $-\delta \frac{w}{|w|^2}\text{sign}(\x)\in B_k^{\mu}$ for $\delta>0$ small enough, we get a contradiction. Lastly, if $\|Z^\ast\|_{L_{\mu}^p(\R^d)}<k$ then we have  
\begin{align*}
u(k,w)&
= \E_\mu[U(\x+\x\langle w, X^\ast+Z^\ast\rangle)]\\
&= \E_\mu[U(\x+\x\langle w, X^\ast-\mathrm{sign}(x_0)|Z^\ast| \frac{w}{|w|}\rangle)]\\
& = \E_\mu[U(\x+\x\langle w, X^\ast\rangle -|\x||Z^\ast||w|)]\\
&>\E_\mu[U(\x+\x\langle w, X^\ast\rangle -|\x|\frac{k}{\|Z^\ast\|_{L_{\mu}^p(\R^d)}}|Z^\ast||w|\rangle)].
\end{align*}
and $\|\frac{k}{\|Z^\ast\|_{L_{\mu}^p(\R^d)}} |Z^\ast| \|_{L_{\mu}^p(\R^d)}=k,$ 
leading to another contradiction. This concludes the proof.
\end{proof}

\begin{lemma}
Fix $w\neq 0$ and assume that $U$ is strictly increasing and satisfies Assumption \ref{ass:2}.  
Assume also that  $\langle \frac{w}{|w|}, X\rangle_*\P$ is absolutely continuous wrt. the Lebesgue measure on $\R$. Then 
\begin{align*}
u(k,w)=\inf_{\Q\in B_k(\P)} E_\Q[U(\x+\x\langle w, X\rangle)]
\end{align*}
has a unique minimizer.
\end{lemma}
We conjecture that the same claim is true without the assumption that $\langle \frac{w}{|w|}, X\rangle_*\P$ is absolutely continuous wrt. the Lebesgue measure on $\R$. We leave this for future research.
\begin{proof}
We first show that \eqref{eq:prob} can be rewritten as a one-dimensional optimization problem. 
Let $\P^w := \langle\frac{w}{|w|}, X\rangle  _*\P.$ 
We claim that 
\begin{align}\label{eq:1dim}
u(k,w)=\inf_{\Q\in B_k(\P)} \E_\Q[U(\x+\x\langle w, X\rangle)]= \inf_{\Q\in B_k(\P^w)} \E_\Q[U(\x+\x|w|X)].
\end{align}
Indeed, Lemma \ref{lem:villani} shows that 
\begin{align*}
\inf_{\Q\in B_k(\P)} \E_\Q[U(\x+\x\langle w, X\rangle)]&=\inf_{X^\mu\sim \P, \, Z^\mu \in B^\mu_k } 
\E_{\mu} [ U(\x+\x\langle w,X^\mu\rangle + \x\langle w, Z^\mu\rangle )].
\end{align*}
For any $X^\mu$ such that $X^\mu\sim \P$ under $\mu$, define $X^w:=\langle \frac{w}{|w|}, X^\mu\rangle.$ 
Then $X^w \sim \P^w$ under $\mu.$ Clearly, we have that 
\begin{align*}
\langle w,X^\mu\rangle + \langle w, Z^\mu\rangle = |w| \Big(X^w + \langle \frac{w}{|w|}, Z^\mu\rangle\Big) \mbox{ and } |\langle \frac{w}{|w|}, Z^\mu\rangle|\le |Z^\mu|
\end{align*}
for any $Z^\mu\in L_{\mu}^p(\R^d)$.
On the other hand for any $Y^\mu \in L_{\mu}^p(\R)$ we have
\begin{align*}
|w|(X^w+Y^\mu)= \langle w ,X^\mu\rangle + \langle w, \frac{w}{|w|}Y^\mu \rangle\quad \text{and}\quad  |\frac{w}{|w|}Y^\mu|=|Y^\mu|.
\end{align*}
This shows 
\begin{align*}
&\inf_{X^\mu\sim \P, \, Z^\mu \in B^\mu_k } \E_{\mu} [ U(\x+\x\langle w,X^\mu\rangle + \x\langle w, Z^\mu \rangle )]\\
&\qquad= \inf_{X^w\sim \P^w,\|Y^\mu\|_{L_\mu^p(\R)}\le  k} \E_{\mu} [ U(\x+\x|w| (X^w+Y^\mu)\rangle )]\\
&\qquad = \inf_{\Q\in B_k(\P^w)} \E_\Q[U(\x+\x|w|X)],
\end{align*}
where the last equality follows from similar arguments as in Lemma \ref{lem:villani}. So,  \eqref{eq:1dim} is proved. 
From Lemma  \ref{lemmex} an optimizer $\Q^\ast \in B_k(\P)$ of  \eqref{eq:prob} exists, so an optimizer $\Q^\ast \in B_k(\P^w)$ for \eqref{eq:1dim} also exists.  
Assume now that there exist $\Q^1\neq \Q^2 \in B_k(\P^w)$ attaining the infimum. Then clearly also the probability measure $(\Q^1+ \Q^2)/2$
attains the infimum, but 
\begin{align*}
\bar{\Q}:=\frac{1}{2} (\Q^1+ \Q^2) \in \text{int}( B_k(\P^w)).
\end{align*}
Indeed, take optimal transport plans $\pi^1$ and $\pi^2$ for $(\P^w,\Q^1)$ and $(\P^w,\Q^2)$. As $\P^w$ is absolutely continuous with respect to the Lebesgue measure, we know by \cite[Theorem 2.9]{santambrogio2015optimal}, that the unique optimal coupling of $\P^w$ and $\bar{\Q}$ for $\mathcal{W}_p$ is supported on a graph of a function. As $(\pi^1+\pi^2)/2 \in \Pi(\P^w,\bar{\Q})$ is not supported on a graph of a function (recalling that $\Q^1\neq \Q^2$), we conclude that it is not optimal. Thus 
\begin{align*}
\mathcal{W}_p( \P^w, \bar{\Q})^p< \frac{1}{2} ( \mathcal{W}_p(\P^w,\Q^1)^p+\mathcal{W}_p(\P^w,\Q^2)^p)\le k^p,
\end{align*}
so that indeed $\bar{\Q}\in \text{int}( B_k(\P^w))$ as claimed. As $U$ is strictly increasing, it then follows that $\Q^1$ and $\Q^2$ do not attain the infimum (using an argument similar to $Z^\ast\neq 0$ in the proof of Lemma \ref{lem:Z}), a contradiction.
\end{proof}

\bibliographystyle{plainnat}
\bibliography{bib}  %

\end{document}